\documentclass[a4paper,10pt,twoside]{article}

\usepackage[utf8]{inputenc}
\usepackage[colorinlistoftodos,prependcaption,color=yellow,textsize=tiny]{todonotes}

\usepackage{amsmath}  
\usepackage{mathrsfs}  
\usepackage{amsthm}
\usepackage{amsfonts}
\usepackage{amssymb, latexsym}
\usepackage{mhequ}
\usepackage{verbatim}

\usepackage{mathptmx}
\DeclareMathAlphabet{\mathcal}{OMS}{cmsy}{m}{n} 

\usepackage{graphicx}
\usepackage{color}
\usepackage{xcolor}
\usepackage{tikz}
\usetikzlibrary{calc,intersections,through,backgrounds}
\usetikzlibrary{decorations.pathreplacing}
\usetikzlibrary{shapes}

\usepackage{float} 
\usepackage{enumerate}
\usepackage[margin=1.5in, marginparwidth=1.2in]{geometry}

\usepackage[pdfauthor={P.~Gassiat, B.~Gess, P-L.~Lions, P.~E.~Souganidis}, pdftitle={Long-time behaviour of stochastic Hamilton-Jacobi equations}, pdftex]{hyperref}
\hypersetup{
  colorlinks=false,
  pdfborder={0 0 0}
  }
\usepackage{url}

\usepackage{tocloft}
\usepackage{fancyhdr}

\usepackage[toc,page]{appendix}

\fancypagestyle{plain}{

  \fancyhf{}
  \fancyfoot[L]{\footnotesize \today \, }
  \fancyfoot[C]{\thepage}
}

\usepackage[colorinlistoftodos,prependcaption,color=yellow,textsize=tiny]{todonotes}

\makeatletter

\usepackage{xcolor} 
\usepackage{hyperref}\hypersetup{linktoc=all,colorlinks=true} 

\usepackage{amsfonts}
\usepackage{color}

\newcommand{\blue}[1]{{ #1}}

\newcommand*{\vs}{\vskip.05in}
\newcommand{\R}{ {\mathbb{R}} }
\newcommand{\E}{ {\mathbb{E}} }
\renewcommand{\P}{ {\mathbb{P}} }

\newcommand{\ep}{\varepsilon}

\renewcommand{\and}{\quad\textrm{ and }\quad}
\renewcommand{\S}{\mathcal{S}}
\renewcommand{\P}{\mathbb{P}}
\renewcommand{\a}{\alpha}
\renewcommand{\b}{\beta}

\newcommand{\T}{\mathbb{T}}

\newcommand{\ve}{\varepsilon}

\newcommand{\Div}{\textnormal{Div}}
\newcommand{\osc}{\textnormal{osc}}

\newcommand{\N}{\mathbb{N}}
\newcommand{\Z}{\mathbb{Z}}

\newcommand{\Ec}{\mathcal{E}}
\newcommand{\Pc}{\mathcal{P}}

\theoremstyle{theorem}
\newtheorem{theorem}{Theorem}[section]
\newtheorem{example}[theorem]{Example}
\newtheorem*{thm*}{Theorem}
\newtheorem{lemma}[theorem]{Lemma}
\theoremstyle{definition}
\newtheorem{definition}[theorem]{Definition}
\theoremstyle{assumption}

\theoremstyle{proposition}
\newtheorem{proposition}[theorem]{Proposition}
\theoremstyle{corollary}
\newtheorem{corollary}[theorem]{Corollary}
\theoremstyle{remark}
\newtheorem{remark}[theorem]{Remark}

\numberwithin{equation}{section}

\makeatother

\title{Long-time behaviour of stochastic Hamilton-Jacobi equations}

\author{Paul Gassiat, Benjamin Gess, Pierre-Louis Lions, Panagiotis E. Souganidis}

\date{\today}

\fancypagestyle{plain}{

  \fancyhf{}
  \fancyfoot[L]{\footnotesize \today \, }
  \fancyfoot[C]{\thepage}
}

\newlength{\bibitemsep}
\newlength{\bibparskip}
\let\oldthebibliography\thebibliography
\renewcommand\thebibliography[1]{%
  \oldthebibliography{#1}%
  \setlength{\parskip}{\bibitemsep}%
  \setlength{\itemsep}{\bibparskip}%
}

\renewenvironment{abstract}
{
\begin{center}
\begin{minipage}{.9\textwidth}\small\textbf{Abstract}\noindent
}
{
\end{minipage}
\end{center}
}

\newenvironment{keywords}
{
\begin{center}
\begin{minipage}{.9\textwidth}\small\textbf{Keywords}:\noindent
}
{
\end{minipage}
\end{center}
}

\newenvironment{msc}
{
\begin{center}
\begin{minipage}{.9\textwidth}\small\textbf{MSC 2020}:\noindent
}
{
\end{minipage}
\end{center}
}

\begin{document}

\maketitle

\begin{abstract}
The long-time behavior of stochastic Hamilton-Jacobi equations is analyzed, including the stochastic mean curvature flow as a special case. In a variety of settings, new and sharpened results are obtained. Among them are  (i)~a regularization by noise phenomenon for the mean curvature flow with homogeneous noise which establishes that the inclusion of noise speeds up the decay of solutions, and (ii)~the long-time convergence of solutions to spatially inhomogeneous stochastic Hamilton-Jacobi equations. A number of motivating examples about nonlinear stochastic partial differential equations are presented in the appendix. 
\end{abstract}


\begin{msc} 60H15, 65M12, 35L65.
\end{msc}

\tableofcontents

\section{Introduction}

\subsection{The general problem}  This work is a contribution towards the development of a systematic understanding of the large time behavior of pathwise solutions of nonlinear  first- and second-order  stochastic partial differential equations (SPDEs for short) of the general form 
\begin{equation}\label{eq:general_SPDE}
du=F(Du,D^{2}u) dt +\sum_{i=1}^{m}H^{i}(x,Du)\circ d{B}^{i}\ \ \text{in } \ \ \mathbb{T}{}^{d}\times\R_{+},
\end{equation}
and
\begin{equation}\label{eq:ph-spde.1}
dv=\Div(F'(v)D v) dt+ \sum_{\substack{1\le i\le m\\ 1\le j \le d}}\partial_{x_{j}}(H^{i,j}(x,v))\circ{dB^{i}}\ \ \text{in } \ \ \T^{d}\times \R_+,
\end{equation}
with $B^{i}$ being independent Brownian motions, $m\in\N$,  $F,H^{i},H^{i,j}$ nonlinear functions, $\mathbb{T}{}^{d}$ the $d$-dimensional torus and $\R_{+}=(0,\infty)$. 
\vs
The solutions to \eqref{eq:general_SPDE} and \eqref{eq:ph-spde.1} are obtained as uniform limits of solutions to the corresponding ``deterministic'' problems with smooth paths approximating the Brownian motions.  
This fact is signified  in \eqref{eq:general_SPDE}  and \eqref{eq:ph-spde.1} by the use of the Stratonovich notation ``$\circ$''. 
\vs
The analysis of the long-time behavior of solutions to \eqref{eq:general_SPDE} and \eqref{eq:ph-spde.1} entails several challenges including the possible degeneracy of the dissipation $F$, and the spatial inhomogeneity of the Hamiltonians $H^{i},H^{i,j}$. 
\vs

SPDEs like \eqref{eq:general_SPDE} and \eqref{eq:ph-spde.1} appear in a variety of applications, including mean field systems with common noise, fluctuating geometric partial differential equations (PDEs for short) such as the stochastic mean curvature flow, and non-equilibrium thermodynamics. We refer to the appendix for more details. 
\vs

Throughout the paper, when dealing with \eqref{eq:general_SPDE}  we work with the so-called stochastic (viscosity) solutions to \eqref{eq:general_SPDE} which were 
developed in a series of papers by the third and fourth authors, see, for example, Lions and Souganidis \cite{LS98,LS98a,LS18+,Li.So2020b,LSS20} as well as Souganidis \cite{So.19} for an extensive overview. The well-posedness of solutions to \eqref{eq:ph-spde.1} follows in many case from the theory of stochastic kinetic solutions developed by Lions, Perthame and Souganidis \cite{Li.Pe.So2013,Li.Pe.So2014}, Gess and Souganidis \cite{Ge.So2015,Ge.So2017}, and its extension to parabolic-hyperbolic SPDEs by Gess and Souganidis \cite{Ge.So2017-2} and Fehrman and Gess \cite{Fe.Ge2019,Fe.Ge2021}. We refer again to \cite{So.19} for an overview.
\vs

In all of the examples we consider here, the typical result  is that, in the long time limit,  the solutions of the SPDEs converge almost surely (a.s.\ for short) to stationary solutions. In the spatially homogeneous case, these are random constants. In some problems, while similar behavior is also observed for the deterministic version of the equations, the presence of stochasticity is shown to accelerate the convergence. In others, although it is known that, in general, solutions to  the deterministic analogues of the problems do not converge, the effect of the stochastic perturbation is such that the deterministic obstacles can be  overcome and the solutions relax to constants. 
\vs

The acceleration of the convergence to constants in the presence of random perturbations  can be explained heuristically as follows. In the examples we consider here, the decrease of the oscillations due the second-order diffusion in \eqref{eq:general_SPDE} decays for large values of the gradient $Du$. Therefore, for large initial gradient $Du$, the decrease of the  oscillations due to the stochastic part of \eqref{eq:general_SPDE} dominates, and yields  an accelerated rate of convergence, compared to the deterministic problem.  The derivation of such quantitative rates of convergence due to  the stochastic parts of the equations is one of the main contributions of the present work. As a main application of these, we prove an increased decay of oscillations in the stochastic mean curvature flow due to the stochastic fluctuations.
\vs

\subsection{The main results} We describe in an informal way, that is, without stating precise theorems, the 
main results obtained in this work, which can be classified in the following  three cases: (i)~any  space dimension and multiple spatially homogeneous noises, (ii)~any space dimension, spatially inhomogeneous noise, and vanishing drift $F$, and (iii)~one spatial dimension and spatially homogeneous noise.

\subsubsection*{Any space  dimension and multiple spatially homogeneous noises}

We consider here SPDEs of the form
\begin{equation}
du = F(Du,D^2 u)dt + \sum_{i=1}^m H_i(Du)  \circ{dB^i}\ \ \text{in } \ \ \mathbb{T}{}^{d}\times\R_{+},\label{eq:intro-multid-homogeneous}
\end{equation}
and present a result yielding  the convergence to constants for large times. The argument is qualitative and does not give a convergence rate. It applies, however, to general settings, and,  in particular, does not impose any restrictions on the dimension.

\vs
Given  $\alpha, \beta^1, \ldots, \beta^m \in \overline \R_+$, let  $\big(S_{\alpha F + \sum_{i=1}^K \beta^i  H_i}(t)\big)_{t\geq 0}$ denote the solution operator, that is the nonlinear semigroup, associated to the deterministic problem
\begin{equation}\label{det}
u_t = \alpha F(Du,D^2 u)+ \sum_{i=1}^m  \beta^i H_i(Du) \ \ \text{in } \ \ \mathbb{T}{}^{d}\times\R_{+}. 
\end{equation}

\begin{theorem}[See Theorem \ref{takis81} below] 
Assume that $F$ is continuous and  non-decreasing in $D^2 u$, $F(0,0)=0$, each $H_i$ is the difference of two convex functions with $H_i(0)=0$ for $i=1,\ldots, m$, and  there exist $\alpha_n, \beta^1_n, \ldots, \beta^m_n\in [0,\infty)$ such that, for all $v \in C(\T^d)$, 
\begin{equation} \label{eq:intro-asn4}
 \lim_{n \to \infty}  S_{\alpha_n F + \sum_{i=1}^m \beta^i_n  H_i}(1)(v) = \mbox{constant}.
 \end{equation} 
Then, almost surely, every solution $u(\cdot,t)$ to \eqref{eq:intro-multid-homogeneous} converges, as $t \to \infty$,  uniformly to a constant. 
\end{theorem}

The heuristic reason behind this result is the fact that, if some combination of the diffusion and noise coefficients of the equation forces convergence, then, by elementary properties of Brownian paths, a.s.~and after some time the Brownian motions will stay close to this combination for a long period of time, after which the solution will be close to a constant. 
\vs 
Note that  \eqref{eq:intro-asn4} includes examples of  $F,$ $H$ for which the separate dynamics $S_{F}$, $S_{H}$ do not converge to constants. In this sense, the result includes an effect of stabilization by noise, where global asymptotic stability is produced by the inclusion of the stochastic fluctuations $\sum_{i=1}^m H_i(Du)  \circ{dB^i}$.
\vs
In particular, the result can be applied to the stochastic mean curvature equation in multiple dimensions implying the convergence of the solution to a constant. In a non-quantitative form, this extends the results obtained in Dabrock,  Hofmanov\'{a} and R\"{o}ger \cite{DHR21} to non-Lipschitz continuous initial data.

\subsubsection*{Any  space dimension, spatially inhomogeneous noise and vanishing drift}

The analysis of the long-time behavior of solutions to \eqref{eq:general_SPDE} with spatially inhomogeneous noise is completely open. Motivated from this, we consider here the simpler case of SPDEs of the type 
\begin{equation}
du=H(x,Du)\circ{dB}\ \ \text{on } \ \ \mathbb{T}{}^{d}\times\R_{+},\label{eq:intro-inhomogeneous}
\end{equation}
with Hamiltonian $H=H(x,p)$ convex in the $p$-variable. 
\vs
The  deterministic version of the problem, that is, $u_t=H(x,Du)$,  has attracted much attention in the literature; see, for example, Fathi \cite{F98}, Namah and Roquejoffre \cite{NR99}, Barles and Souganidis \cite{BS00}, and Barles, Ishii and Mitake  \cite{BIM13}. These works provide conditions under which the deterministic semigroups converge, in the sense that there exists $c\in\R$ such that, for all $u\in C(\T^{N})$  and uniformly as $t\to \infty$, 
\[
\lim_{t\to \infty}\|S_{H}(\pm t)u\pm ct -\phi^{\pm}\|=0,
\]
where $\phi^{\pm}$ solves, in the viscosity sense, 
\begin{equation}
\pm
(c+H)(x,D\phi^{\pm})=0\mbox{ in }\T^{d}.\label{eq:E+}
\end{equation}

We obtain the following analogous result in the stochastic case.

\begin{theorem}[See Theorem \ref{thm:Hconvfwd} below]
Assume that $H=H(x,p)$ is strictly convex in $p$. Then, there exist $\phi^\pm:\T^d\to \R$ satisfying \eqref{eq:E+}, and  a global-in-time, statistically stationary, solution $\psi:\R\times\T^{d}\to\R$ to
\begin{equation}
d\psi={(c+H)}(x,D\psi)\circ{dB} \ \ { on }\ \ \T^{d}\times \R,
\end{equation}
satisfying 
\[\phi^{-}\leq\psi(\cdot, t)\leq\phi^{+} \ \ \text{ for all} \ \  t\in\R_+ \quad \text{and} \quad \limsup_{t\to\infty}\psi(\cdot,t)=\phi^{+},\;\;\;\liminf_{t\to\infty}\psi(\cdot, t)=\phi^{-},\]  
so that for the solution $u$ to \eqref{eq:intro-inhomogeneous} we have 
\[
\lim_{t\to  \infty} \|u(\cdot,t) + cB(t) - \psi(\cdot,t)\|=0.
\]
\end{theorem}

A particular  case is  $H(x,p)= a(x)\sqrt{1+|p|^2}$ with $a(\cdot)> 0$, which corresponds to the motion of a graph interface driven by isotropic noise, see Section \ref{sec:front_propagation} below. 
\vs

We note that it is possible to replace  the strict convexity of $H$ by  only convexity  and the assumptions that the 
deterministic semigroups converge and the solutions to \eqref{eq:intro-inhomogeneous} are equicontinuous in $x$. 
\vs 

We comment next on the proof of the theorem above. One of the main difficulties in the analysis of stochastic Hamilton-Jacobi equations comes from the irreversibility of the dynamics. For example,  even for equations driven by $1$-dimensional noise as above, $B(t)=0$ does not imply that $u(\cdot,t)=u(\cdot,0)$, even though this would be the case if the equation was satisfied in the classical sense. Nevertheless, it was  shown in Gassiat, Gess, Lions and Souganidis \cite{GGLS20} that in the case of convex Hamiltonians a certain monotonicity of $u$ with respect to the driving signal allows the proof of cancellations at the level of the PDE solution $u$; see also, Hoel, Karlsen, Risebro, Storrosten \cite{HKRS20,HKRS18}. Repeated applications of this observation, combined with properties of the deterministic equations, allow us to obtain the existence of  the stationary solutions $\psi$, and to prove that all solutions are close to them for large times.

\subsubsection*{Spatially homogeneous noise and $d=1$}

We prove new quantitative decay estimates and large time convergence for SPDEs of the form
\begin{equation}
du= \partial_{x}(F(\partial_{x}u)) dt +H(\partial_{x}u)\circ{dB}\ \ \text{on } \ \ \mathbb{T}\times\R_{+}. \label{eq:intro-1d-homogeneous}
\end{equation}
For the precise assumptions on $F,H$ see Section \ref{sec:1d_homogeneous} below. 
\vs 
The proof relies on a novel combination of entropy inequalities with the nonlinearity of the Hamiltonian $H$ to deduce quantitative decay estimates for the derivative $\partial_{x}u$. As a first main result, in Section \ref{sec:homogeneous}, we derive quantitative estimates relying purely on the stochastic part in \eqref{eq:intro-1d-homogeneous}, thus giving estimates uniform  in $F$.

\begin{theorem}[See Theorem \ref{thm:quantifiedbetter} below] \label{intro-thm:quantifiedbetter} 
Assume that
$H=H(x,p)$  grows super-linearly in $p$ with rate $q\geq 1$, and is the difference of two convex functions $H_1,H_2$, with $H  \geq \alpha H_1$ for some $\alpha \in (0,1]$, and $F\in C^1(\R)$ is non-decreasing. Then, for almost every Brownian path, there exist positive constants $c$ and $C(B)$ depending only on $q$ and $\alpha$ such that, if $q=1$, then the solutions $u$ to \eqref{eq:intro-1d-homogeneous} satisfy
\begin{equation} 
\|u_x(\cdot,T)\|_{L^1} \leq C(B) (e^{-c T} \osc(u_0) +  1),
\end{equation}
and, if $q>1$, 
\begin{equation}
\|u_x(\cdot, T)\|_{L^q}  \leq C(B)  \left(  T^{-\frac{1}{2(q-1)}}  + 1 \right) .
\end{equation}
\end{theorem}

Note, that, if $q>1$, the bound on $ \|u_x(\cdot, T)\|_{L^q}$  is uniform with respect to the initial condition.

\vs
In Section \ref{sec:dissipation}, we derive estimates that in contrast rely entirely  on the dissipation $F$, and are uniform with respect to the noise $H$. For the sake of brevity, we discuss here only the results in the third case analyzed in Section \ref{sec:interplay}, where the  interplay of diffusion and noise is exploited to derive improved decay estimates. 
\vs
We consider  the stochastic mean curvature equation for a one-dimensional graph
\begin{equation}
du= \frac{\partial_{xx} u}{\sqrt{1+|\partial_x u|^{2}}}dt+\sqrt{1+|\partial_x u|^{2}}\circ dB \ \ \text{in } \ \ \mathbb{T}\times\R_{+}.\label{eq:intro-smcf}
\end{equation}
The analysis of the long-time behavior of solutions to \eqref{eq:intro-smcf} is challenging in view of the degeneracy of the ellipticity of the mean-curvature operator for large values of the gradient  $\partial_x u$.  Indeed, it was shown in Colding and Minicozzi~\cite{CM04} that in the deterministic setting, that is, for 
\[u_t= \frac{\partial_{xx} u}{\sqrt{1+|\partial_x u|^{2}}},\]  this degeneracy leads to limitations on the speed of convergence of solutions. More precisely, as described in Remark \ref{rmk:mcf_limit}, it is possible to find  solutions $u^R$ with $\osc(u^R(\cdot, 0))=R$ and
\begin{equation}\label{eq:intro-CM}
  \liminf_{R\to\infty}\osc(u^R(\cdot,R))\ge 1.
\end{equation}
\vs

In contrast, we prove that the inclusion of noise in the stochastic mean curvature equation has a ``regularization by noise'' effect in the sense that it improves the dependency of the decay of solutions on the initial condition. The heuristic reason for this improvement is that, for large initial gradient $\partial_x u(\cdot, 0)$, the decay of oscillations caused by the mean curvature operator in \eqref{eq:intro-smcf} becomes small, and, hence,  the decay of oscillations due to the stochastic part in \eqref{eq:intro-smcf} dominates. 

\begin{theorem}[See Theorem \ref{takis68} below]
Let $u$ be the solution to \eqref{eq:intro-smcf} driven by  Brownian motion $B$ with initial condition $u_0$, and set  $\tau(u_0,B) := \inf \left\{t \geq 0: \;\; \osc(u(\cdot,t)) \leq 2 \right\}$. 
Then there exists a deterministic constant $c>0$, and, for almost every Brownian path, constants  $C(B), \rho(B) \in \R_+$,   
such that
\begin{equation} \label{eq:intro-tau}
\tau(u_0,B) \le c \log(1+\osc(u_0)) + C(B), 
\end{equation}
and
\begin{equation} \label{eq:intro-uxSmcf}
\|u_x(\cdot,t)\|_{\infty} \leq C(B) \osc(u_0) e^{-c t}  \ \ \text{for all} \ \  t \geq \tau(u_0,B) + \rho(B). 
\end{equation}
\end{theorem}

In particular, these estimates provide a quantified and path-by-path improvement of the qualitative results obtained in Es-Sarhir and von Renesse \cite{ER12}.

\vs 

Notably, \eqref{eq:intro-uxSmcf} implies that eventually the solution regularizes, that is, it  becomes Lipschitz continuous, and, moreover,  
  $$\|u_x(t)\|_{\infty} \le 1 \ \ \text{for all} \ \ t \ge C\log(\osc(u_0)+1)+C(B),$$
a fact that shows that the inclusion of noise improves the relaxation of the initial oscillation from the superlinear time scale found in \eqref{eq:intro-CM}  in the deterministic case to a logarithmic one. 
\vs
A second application and example of increased speed of convergence concerns  stochastic Hamilton-Jacobi equations (sHJ for short)  with polynomial nonlinearities, that is, for $\alpha, \beta > 1$ and $(u_x)^{[\alpha]}:=u_x|u_x|^{\alpha-1},$
\begin{equation}\label{eq:model_2}
       d u = \partial_{x}(u_x)^{[\alpha]}\,dt+ |u_x|^{\beta} \circ dB\ \ \text{in } \ \ \mathbb{T}\times\R_{+}.
\end{equation} 
\vs
In the deterministic version of \eqref{eq:model_2}, that is, for the PDE $u_t= \partial_{x}(u_x)^{[\alpha]}$, the degeneracy of the diffusion at $u_x = 0$ yields that the  decay of $\|u_x(t)\|_\infty$ is of order  $t^{-\frac{1}{\a-1}}$ which is slow for $\alpha$ large--this bound can be easily derived from solutions with separated variables. 
\vs

In contrast, it follows from the results obtained in this work that a higher order of decay is caused by the stochastic part if $2\beta< 1+a$. Indeed, in this range, the decay of oscillations due to the stochastic part in \eqref{eq:model_2} dominates the decay caused by the degenerate diffusion. 

\begin{theorem}[See Example \ref{eq:polynomial_SPDE} below]\label{intro-eq:polynomial_SPDE}
For  $\alpha, \beta > 1$, let $u$ be the solution to \eqref{eq:model_2}  with initial condition $u_0$ and  $B$ a Brownian motion. There exists a constant $C(B)\in \R_+$, such that, for all $t>0$, 
     \begin{equation} 
     \|u_x(\cdot,t)\|_{L^\beta} \leq C(B) t^{-\frac{1}{2(\beta-1)}}.
     \end{equation}
     and
     $$ \|u_x(
     \cdot, t) \|_\infty  \le C(B) t^{-\frac{\beta}{2(\beta-1)(\alpha+\beta-1)}}.$$
\end{theorem}
Note that, when $\beta$ is close to one,  this estimate significantly exceeds the optimal deterministic decay of order  $t^{-\frac{1}{\a-1}}$.

\subsection{Organization of the paper}
The paper is organized as follows.  In section~\ref{qc} we present an argument that yields the convergence to constants for a large class of SPDEs but without a rate of convergence. Section~\ref{sec:inhomogeneous} is devoted to the study of the long-time behavior of solutions to stochastic Hamilton-Jacobi equations with inhomogeneous Hamiltonians that are  convex in the gradient. Section~\ref{sec:1d_homogeneous}, which  is about the long-time behavior of parabolic-hyperbolic SPDEs with spatially homogeneous noise, consists of following three parts: (i)~ convergence due to the stochastic fluctuations, (ii)~convergence due to the dissipation, and (iii)~convergence due to the interaction between the stochastic fluctuations and dissipation. In each subsection, we present several results and examples. In Section~\ref{sec:open}, we list  a number of questions that  left open by the present work and which we believe to be of interest. Finally, in the Appendix we discuss a number of motivating and concrete examples of SPDEs to which our results apply. The appendix can also be considered as an introduction of the scope of a large number of nonlinear SPDEs with multiplicative stochastic dependence arising in concrete applications, and, in addition, provides the motivation for the concrete problems addressed earlier in the paper.

\subsection{Notation} We write $\T^d$ for the $d-$dimensional torus and, when $d=1$,  $\T$ instead of $\T^1$. Throughout the paper $\|w\|_X$ stands for the norm of $w\in X$. When $X=L^\infty$, we simply write $\|w\|$.  We write  $BV$ for the set of functions of bounded variation and  $\text{BUC}(X)$ for the set bounded and uniformly continuous  functions on $X$. The oscillation of a function $w$ is denoted by $osc(w)$. 
For an integrable function $F:\R\to \R$, we set  $[F](v):=\int_{0}^{v}F(r)dr$. Given two random variables $X$ and $Y$, $X\overset{d}=Y$ means that $X$ and $Y$ have the same distribution. Moreover, we write $\mathcal{L}(X \; | \mathcal{W})$ for the law of the random variable conditioned upon the $\sigma-$algebra associated with a Brownian path $W$. 

\subsection{Acknowledgments} BG acknowledges support by the Max Planck Society through the Research Group "Stochastic analysis in the sciences". This work was funded by the Deutsche Forschungsgemeinschaft (DFG, German Research Foundation) - SFB 1283/2 2021 - 317210226. PES  was partially supported by the National Science Foundation grants  DMS-1900599 and DMS-2153822, the Office for Naval Research grant N000141712095 and the Air Force Office for Scientific Research grant FA9550-18-1-0494.

\section{Qualitative convergence for homogeneous equations} \label{qc}

We consider general SPDEs of the type 
\begin{equation} \label{eq:sec4}
du = F(Du,D^2 u) dt  + \sum_{i=1}^m H_i(Du)\circ dB^i  \ \  \mbox{ in } \ \  \T^d \times \R_+,
 \end{equation}
 where $B=(B^1,\ldots, B^m)$ is a $m$-dimensional Brownian motion and 
\begin{equation}\label{takis80}
\begin{split}
&F=F(p,X)  \ \ \text{ is continuous and  non-decreasing in $X$ \  and  \ $F(0,0)=0$,}\\[1.2mm]
& \text{$H_i$ is the difference of two convex functions, \ and \ $H_i(0)=0$ for $i=1,\ldots, m$, }\\[1.2mm]
\end{split}
\end{equation}
and present a simple argument that yields  convergence to constants. The proof  is qualitative and does not give a convergence rate but it can be applied in general settings, and,  in particular, does not impose any restrictions on the dimension.
\vs The results extend to the more general equations
\begin{equation}\label{takis820}
du = F(Du,D^2 u) dt  + \sum_{i=1}^m H_i(Du)\circ d\xi^i  \ \  \mbox{ in } \ \  \T^d \times \R_+,
\end{equation}
where $\xi=(\xi^1,\ldots,\xi^m)$ is any  continuous, stationary, ergodic process as long as its restriction to intervals have full support. This is, for instance, the case of fractional Brownian motion.

\vs 
We recall from the introduction that, for $v\in C(\T^d)$ and $t\geq 0$, $S_{\alpha F + \sum_{i=1}^m \beta^i  H_i}(t)v$ denotes the solution to the deterministic initial value problem 
\[ V_t=\alpha F(DV, D^2V)  + \sum_{i=1}^m \beta^i  H_i(DV) \ \  \mbox{ in } \ \  \T^d \times \R_+, \ \ \ V(\cdot,0)=v,\] 
and we  assume that 
\begin{equation} \label{eq:asn4}
\begin{split}
&\text{there exist $\alpha_n, \beta^1_n, \ldots, \beta^m_n$ such that, for all $v \in C(\T^d)$,}\\[1.2mm]
 &\quad \lim_{n \to \infty}  S_{\alpha_n F + \sum_{i=1}^m \beta^i_n  H_i}(1)(v) = \mbox{constant}.
 \end{split}
 \end{equation}

The qualitative long-time result is stated next. 

\begin{theorem}\label{takis81}
Assume \eqref{takis80} and  \eqref{eq:asn4} and  $B=B(\omega)$ is a $m$-dimensional Brownian motion sample path. Let  $u \in C(\T^d \times [0,\infty))$ be a solution to \eqref{eq:sec4}. Then, almost surely in $B$, $u(\cdot,t)$ converges,  as $t \to \infty$ and uniformly in $\T^d$,  to a constant.
\end{theorem}

Before we proceed with the proof, a number of remarks are in order. 
\vs

As far as \eqref{eq:asn4} is concerned, typically,  it can be  assumed that either  $\alpha_n=n$ and $\beta^i_n \equiv 0$  or $\alpha_n =0$ and  $\beta^i_n=n$, that  is, that solutions to deterministic equations corresponding to either $F$ or one of the $H_i$ converge for large times  to a constant. 
\vs

Examples of the former are 
$ F=F(p) \geq 0 \mbox{ with equality if and only if  } p=0,$ and 
$F(p,A) = Tr( a(p) A ) \mbox{ with }  a(\cdot) $ taking values in the set of strictly positive symmetric matrices.
Hence, our results apply to the case 
\[ F(p,A)  =  Tr\left( A  \left(I_d - \frac{p \otimes p}{1+|p|^2}\right)\right) \ \ \text{and} \ \ H(p)= \sqrt{1 +|p|^2}, \]
corresponding to the stochastic mean-curvature flow of a graph, where $I_d$ is the identity matrix,  We therefore recover in a simpler way, but in a different topology,  the result of \cite{DHR21} that the graph becomes asymptotically constant for large time. 
\vs

It is also possible to find examples where $S_F$ and the $S_{H_i}$ taken separately do not converge. Indeed, for instance take $d=m +1 \geq 2$, $F(p) = |p_1|$, $H_i(p)=|p_{i+1}|$. On the other hand,  \eqref{eq:asn4} holds for $\alpha_n=\beta^1_n=\ldots=\beta^d_n=n$.  
\vs

The idea of the proof is classical. It  combines the compactness property in the space variable, the continuous dependence of the solutions on the paths  and the fact  that, in large time intervals, the noise is small. For  additive noise dependence, the problem was investigated by 
Dirr and Souganidis \cite{Di.So2005}.
\vs

Increasing the regularity of the path $\xi$ in \eqref{takis820}, it is possible to weaken  the regularity requirement on $H$. Indeed, when $\xi$ is a   Brownian motion, its trajectories take value in $C^{0,\alpha}([0,T])$ for any $\alpha < \frac 1 2$, and the support of its law in this space still contains all smooth functions. It follows, for example,  from the results of Lions, Seeger and Souganidis \cite{LSS20} that the map $\xi \in C^{0,\alpha} \mapsto u$ is continuous if, for some $\varepsilon>0$,  $H \in C^{1,\varepsilon}$.
\vs
In principle, the proof below is flexible and could extend to $x$-dependent equations as long as the solutions are continuous  with respect to  the driving signal and  equicontinuous  in the space variable. The latter  is delicate and, at the moment, is only known for $x$-dependent equations in the case where $H=H(x,Du)$ is convex and $F\equiv 0$; see \cite{GGLS20}. However, in that case, we have a more precise description of the long-time dynamics which are discussed  in section \ref{sec:inhomogeneous} below.

\begin{proof}[The proof of Theorem~\ref{takis81}]
Throughout the argument, without loss of generality, we assume that $\alpha_n > 0$.
\vs
It follows from the contraction property of the solution operator of \eqref{takis820} with $\xi=(\xi^1,\ldots, \xi^m)$ an arbitrary continuous path (see \cite{So.19}) that the   the map $t\to \osc(u(\cdot,t))$ is non-increasing. 
\vs
Thus to prove the result,  it suffices to show that, for some subsequence $t_n\to \infty$, $osc(u(\cdot,t_n)) \to 0$.
\vs
It follows again from the contraction property of the solutions of \eqref{takis820}, the continuity with respect to the paths (see \cite{So.19})  and the compactness of $\T^d$ that, given $u_0 = u(\cdot,0)$ and $t\geq 0$, the set  
\[ \mathcal{K}(u_0) = \left\{ v(\cdot,t) : \;  \mbox{$v$ solves }\eqref{eq:sec4} \mbox{ for some path } B \in C([0,t], \R^d) \right\} \]
is compact in $C(\T^d)$.
\vs
 Indeed, the comparison principle and the homogeneity of  \eqref{eq:sec4} give that the equation preserves upper and lower bounds and modulus of continuity by a standard comparison argument.  For instance, for the latter point, note that, for all $y\in \R^d$,  if $v$ is a solution to \eqref{eq:sec4}, so does $v(\cdot+y,t)$ and, hence, the comparison principle (see \cite{So.19})  yields 
 \[  \sup_{t \geq 0, x \in \T^d}  \{v(x+y,t) - v(x,t)\}  \leq \sup_{x \in \T^d} \{u_0(x+y) - u_0(x)\}. \]
\vs
The  compactness implies that the convergence in \eqref{eq:asn4} is uniform over $v \in \mathcal{K}(u_0)$, that is,  
\begin{equation} \label{eq:convK}
 \lim_{n \to \infty}  \sup_{v \in \mathcal{K}(u_0)} \osc \left( S_{\alpha_n F + \sum_i \beta^i_n H_i}(1)(v)\right) =  0,
 \end{equation}
 
{ 
Recall that, for any $\alpha>0$, the law of Brownian motion has full support on $C([0,\alpha])$, namely for any continuous function $f : [0,\alpha] \to \R$ with $f(0)=0$, and any $\varepsilon > 0$, it holds that
\[ \P \left( \forall t \in [0,\alpha], \;  \left|  B(t) - f(t) \right|  \leq \varepsilon \right) := p_{\alpha,f ,\varepsilon} > 0. \]
By independence and stationarity of increments, this further implies that, 
\begin{align*}
&\P \left( \exists T \geq 0 :  \forall t \in [0,\alpha], \;  \left|  B(t+T) - B(T)- f(t) \right|  \leq \varepsilon \right) \\
\geq&\; \sup_{N\geq 1} \P \left( \exists 0 \leq k \leq N :  \forall t \in [0,\alpha], \;  \left|  B(t+k \alpha) - B(k \alpha)- f(t) \right|  \leq \varepsilon \right) \\
\geq &\sup_{N \geq 1} \left\{ 1 - (1- p_{\alpha,f ,\varepsilon} )^N\right\} \\
= & 1. 
\end{align*}

Letting now $\varepsilon_n \to 0$ be arbitrary, this implies that, 
almost surely, there exists $T_n$ such that
\[ \sup \left\{  \sum_{i=1}^d\left| B^i(t+T_n) - B^i(T_n) - \beta^i_n t \right|, \;\;\; t \in [0,\alpha_n ] \right\} \leq \varepsilon_n. \]
}
Moreover, in view of the  continuity of the solution map with respect to  $B$ and the fact that $u(\cdot,t)$ is in $\mathcal{K}(u_0)$, we find  that, for some  $\varepsilon'_n \to 0$ depending  on $\varepsilon_n$ and $\mathcal{K}(u_0)$, 
\[ \left\| S_{\alpha_n F + \beta_n H}(1)(u(\cdot,T_n)) - u(\cdot,T_n+ \alpha_n)  \right\|_{\infty} \leq \varepsilon'_n .\]

The assertion above and  \eqref{eq:convK} imply that $ osc(u(\cdot,T_n+\alpha_n)) \to 0$, and, hence, the result.
\end{proof}

\section{Inhomogeneous and convex Hamiltonians}\label{sec:inhomogeneous}

We investigate  here the long time behavior of $\T^d-$periodic solutions to the inhomogeneous Hamilton-Jacobi equation 
\begin{equation} \label{eq:HJstoc}
du + H(x, Du)\circ d\xi = 0 \ \ \text{in} \ \ \T^d\times \R_+ \quad  u(0,\cdot) = u_0 \ \ \text{on} \ \ \T^d,
\end{equation}
with $\xi\in C([0,\infty))$, a special case  being the stochastic Hamilton-Jacobi equation
\begin{equation} \label{eq:HJstoc1}
du + H(x, Du) \circ dB = 0 \ \ \text{in} \ \ \T^d\times \R_+ \quad  u(0,\cdot) = u_0 \ \ \text{on} \ \ \T^d.
\end{equation}

\vs
We denote by $S_H$ the solution operator (semigroup)  associated to the deterministic evolution $v_t + H(x,Dv) = 0$, and, for each $\xi \in C^1([0,\infty))$ and each $v \in C(\T^d)$,  $\S_H^{\xi,[0,t]})(v)$ is the value at time $t$ of the solution to $u_t + H(x,Du) \circ d{\xi}= 0$ and $u(\cdot, 0) = v$.
\vs

Throughout the section we make two assumptions which we state next. Their role was already mentioned in the discussion before the proof of Theorem~\ref{takis81}. 
\vs
We assume that 
\begin{equation}\label{takis90}
H=H(x,p) \mbox{ is continuous on }\T^d \times \R^d \mbox{ and convex in the $p$ variable}, 
\end{equation}
and
\begin{equation}\label{takis91}
\begin{split}
&\mbox{for any compact }K \subset C(\T^d) \ \text{and} \ T>0, \ \mbox{ the family }\\
&\left\{ \S^{\xi,[0,T]}_{H} u: \;u \in K, \; \xi \in C^1([0,T]) \right\} \mbox{ is equicontinuous.} 
\end{split}
\end{equation}

We remark that sufficient conditions for \eqref{takis91} to hold are given Theorem~A.1 in Seeger \cite{Seeger20} and Proposition~2.5 in \cite {GGLS20}. Note that, in the  periodic setting of this paper, it is true that  $H(x,p) \leq \bar{H}(p)$ for a convex $\bar{H}$, and, thus, (2.8) in \cite{GGLS20} is always satisfied. Then for \eqref{takis91} to hold requires additional information like  local controllability in the control problem associated to $H$. Coercivity of $H$ is, for example,  sufficient  but not necessary.

\vs
When $H$ is convex in the gradient, the control representation of $S_{\pm H}$ yields  a useful monotonicity lemma which is stated next. For the proof we refer to \cite{GGLS20} (see also \cite{BCJS99}). 

\begin{lemma} \label{lem:mon}
Assume  \eqref{takis90}. Then,  for all $t \geq 0$, \ \  $S_{-H}(t) S_H(t) \leq Id \leq S_{H}(t) S_{-H}(t).$
\end{lemma}

We further recall that using the lemma above repeatedly allows to prove some monotonicity properties for the solution map, which in particular,  in combination with the compactness assumption \eqref{takis91}, imply that the solution map may be extended to arbitrary continuous $\xi$. 
In \cite{GGLS20}, these results are stated in Corollary~2.4, Corollary~2.8 and  Corollary~2.7, which, for the convenience of the reader, we summarize in the following lemma and proposition. 

\begin{lemma}\label{takis92}[Corollaries 2.4 and Corollary~2.8 in \cite{GGLS20}] 
\label{lem:inf} Assume \eqref{takis90}. Then, for all   $\xi, \zeta \in C^1([0,T])$ such that  $\xi(0)=\zeta(0)$,  $\xi(T)=\zeta(T)$ and $\xi  \leq \zeta$ on $[0,T]$, 
\[ \S_{H}^{\xi;[0,T]} \geq \S_H^{\zeta;[0,T]},
\]
and, for any $\xi \in C^1 ([0,T])$ with $\xi(0)=0$,
\[
 S_{H} \left(\xi(T) - \min_{[0,T]} \xi \right) S_{-H} \left(-\min_{[0,T]} \xi \right) \geq  \S_{H}^{\xi;[0,T]} \geq S_{-H} \left( \max_{[0,T]} \xi - \xi(T)\right) S_{H} \left(\max_{[0,T]} \xi \right).
\]
If $\xi$ is such that $\xi(0) = \inf_{[0,T]} \xi$ and $\xi(T) = \sup_{[0,T]} \xi$, then $ \S_{H}^{\xi;[0,T]} = S_H(\xi(T)-\xi(0))$. Similarly, if $\xi(0) = \sup_{[0,T]} \xi$ and $\xi(T) = \inf_{[0,T]} \xi$, then $ \S_{H}^{\xi;[0,T]} = S_{-H}(-\xi(T)+\xi(0))$.
\end{lemma}
\vs
\begin{proposition}[Corollary 2.7 in \cite{GGLS20}]%
Assume  \eqref{takis90} and \eqref{takis91}. Then, 
\begin{equation}
\mbox{the solution map } \xi \mapsto\mathcal{S}_H^{\xi,[0,T]}(u) \mbox{ admits a unique continuous extension to } \xi \in C([0,T]),
\end{equation}
and the results of Lemma~\ref{takis92} still hold for $\xi,\zeta$ $\in$ $C([0,T])$.
\end{proposition}

To establish   the long-time behavior of \eqref{takis90}, that is, the convergence to constants,  we need to assume that solutions to the deterministic problem, that is, when $\xi(t)=\pm t$, converge in the sense that
\begin{equation}\label{asn:Hinf}
\begin{split}
&\text{there exists $c \in \R$ such that, for each} \; u \in C(\T^d), \text{there exist}\\[1.2mm]
&\text{ $\phi^\pm= \phi^\pm(u)\in C(\T^d)$ such that} \ 
\  \pm c\pm H(x, D\phi^\pm)=0 \ \text{on} \  \T^d \ \ \text{and}\\[1.2mm]
 & S_H(T) u - c T  \underset{T\to \infty}\to \phi^+ \mbox{ and } S_{-H}(T) u + c T  \underset{T\to \infty}\to \phi^- \ \text{uniformly in $\T^d$}.\\
\end{split}
\end{equation}

It is a classical fact in the theory of viscosity solutions that 
the equations $c + H(x,D\phi^+) = 0$ and  $- c - H(x,D\phi^-) = 0$ are not equivalent except, of course, when there exist 
$C^1-$solutions. 
\vs

The fact that $H$ and $-H$ have opposite ergodic constants is true here due to the convexity of $H$ but is not true in general; see, for example, the discussion about this issue in \cite[Section 4.1]{Seeger18}.
\vs
Finally, sufficient conditions for the convergence in \eqref{asn:Hinf} have been obtained by several authors using either  control or PDE arguments. The literature is very long. Here,  we only refer to  \cite{F98}, which assumes that  $H$ strictly convex in the gradient, 
and  \cite{BS00} for the  most general, in the sense that  no convexity of $H$ is required,  convergence result.
\vs


In order to simplify the notation for what follows, we  set $\hat{H} = H + c$, in which case we may assume that $c=0$ and note that the convergence in \eqref{asn:Hinf} may be restated as convergence, for any $u \in C(\T^d)$,  of $S_{\hat{H}}(T)u$ and $S_{-\hat{H}}(T)u$.
\vs

We denote by $\Ec^\pm$ the set of continuous  solutions to $\pm {\hat H}(x, D\phi^\pm)=0 $ and record in the next lemma some of its properties and provide a sketch of their proof.

\begin{lemma}  \label{lem:E+-}
Assume  \eqref{takis90} and \eqref{asn:Hinf}. Then,
\vs
(i)~if $\phi \in \Ec^+$ (resp.  $\phi \in \Ec^-$), then, for each $T \geq 0$, $S_{-\hat{H}}(T)\phi \leq \phi$ (resp.  $S_{\hat{H}}(T) \phi \geq \phi$), and 
\vs
(ii)~the maps
\[ \phi \in \Ec^- \mapsto S_{\hat{H}}(\infty) \phi := \lim_{D \to \infty} S_{\hat{H}}(D) \phi \in \Ec^+ \ \text{and} \  \phi \in \Ec^+ \mapsto  S_{\hat{H}}(-\infty) \phi :=\lim_{D \to \infty} S_{-\hat{H}}(D) \phi \in \Ec^- \]
are inverse of each other.
\end{lemma}

\begin{proof}
The first claim follows immediately from Lemma \ref{lem:mon} and the fact that $\phi \in \Ec^+$ if and only if $\phi = S_{\hat{H}}(T)\phi$ for all $T \geq 0$.
\vs
For (ii),  it suffices to prove that, for each $\phi \in \Ec^+$, $S_{\hat{H}}(\infty) S_{\hat{H}}(-\infty) \phi = \phi$, and this is obtained as follows:   Lemma \ref{lem:mon} yields that  $S_{\hat{H}}(\infty) S_{\hat{H}}(-\infty) \phi \geq \phi$, and (i) above  gives  that $S_{{\hat{H}}}(-\infty) \phi \leq \phi$. We then deduce from the comparison principle that  $S_{\hat{H}}(\infty) S_{\hat{H}}(-\infty) \phi \leq  S_{\hat{H}}(\infty) \phi = \phi$. 

\end{proof}

Following  \cite{FathiBook}, we call a pair $(\phi^+,\phi^-) \in \Ec^+ \times \Ec^-$ such that $\phi^- = S_{-{\hat{H}}}(\infty) \phi^+$ conjugate, and we denote by  $\mathcal{P}$ the set of such pairs.
\vs
{ 
We will frequently use the following monotonicity property.  
\begin{lemma} Let $(\phi^+,\phi^-)\in \mathcal{P}$, and $u \in C(\T^d)$ such that $\phi^- \leq u \leq \phi^+$. Then, for any $\xi \in C([0,T])$, it holds that
\[ \phi^- \leq S_{\hat{H}}^{\xi,[0,T]} u \leq  \phi^+.  \]
\end{lemma}
\begin{proof}
By comparison, it suffices to prove the claim for $u = \phi^-$ and $u= \phi^+$. If $\xi$ is piecewise linear, this is an immediate implication of the following inequalities (for arbitrary $\delta \geq 0$)
\[ S_{-\hat{H}} (\delta) \phi^- = \phi^-  \leq S_{\hat{H}} (\delta) \phi^- ,\;\;\;S_{-\hat{H}} (\delta) \phi^+ \leq \phi^+  =  S_{\hat{H}} (\delta) \phi^+, \]
and the general case follows by density.
\end{proof}
}
The next proposition shows that each conjugate pair is associated to a unique global (in time) solution to $du  + H(x,Du) \circ d{\xi}=0$. 
\vs
We also show, although we will not need it later, that this correspondence is actually one-to-one if 

\begin{equation}\label{asn:reg}
\text{there exists  $T > 0$ \ such that  the set \ $\left\{ S_H(T) u : \;  u \in C(\T^d) \right\}$ is equicontinuous}.
\end{equation}

\begin{proposition} \label{prop:psi}
Assume \eqref{takis90}, \eqref{takis91} and \eqref{asn:Hinf}, and let $\xi \in C(\R;\R)$ be such that 
$\limsup_{t \to - \infty} \xi(t) = - \liminf_{t \to - \infty}  \xi(t)= +\infty.$
Then,
\vs
(i)~for each conjugate pair $(\phi^+,\phi^-)$, there exists a unique solution $\psi :\T^d \times \R\to \R$ to 
\begin{equation} \label{eq:HJglobal}
d \psi  + {\hat{H}}(x,D\psi)\circ d{\xi} = 0 \mbox{ on } \T^d \times \R,
\end{equation}
 such that 
\begin{equation} \label{eq:ineqPsi}
\phi^-(x) \leq \psi(x, t) \leq \phi^+(x) \ \text{ for all} \ t \in \R \ \text{and} \  x \in \T^d,
\end{equation}
and
\vs
(ii)~if, in addition \eqref{asn:reg} holds,  then, conversely, given a solution $\psi$ to \eqref{eq:HJglobal}, there exists a unique $(\phi^+,\phi^-)$ $\in$ $\mathcal{P}$ such that \eqref{eq:ineqPsi} holds.
\end{proposition}

\begin{proof} To prove (i) choose a decreasing sequence $(T^n)_{n\in \N}$ such that $T^n \underset{n\to \infty}\to -\infty$, $T^0=0$, and, for $n \geq 1$,
\[  \xi(T^{2n+1}) = \min_{[T^{2n+2},T^{2n}]} \xi, \;\;\; \xi(T^{2n}) = \max_{[T^{2n+1},T^{2n-1}]} \xi, \]
and
\[ D^n_+ :=    \xi(T^{2n}) - \xi(T^{2n+1}) \underset{n\to \infty} \to+ \infty, \;\;\;\;D^n_- :=    \xi(T^{2n+1}) - \xi(T^{2n+2}) \underset{n\to \infty}\to -\infty,
\]
and  let $\psi^{n,+} : \T^d \times [T^n, +\infty) \to \R$ (resp. $\psi^{n,-}$) be the solution to \eqref{eq:HJglobal} in $\T^d \times [T^n, +\infty)$ with $\psi^{n,+}(T^n,\cdot) = \phi^+$ (resp. $\phi^-$).
\vs
It follows from  Lemma~\ref{lem:inf} and Lemma~\ref{lem:E+-} that $\psi^{n,+}$ is non-decreasing (resp. $\psi^{n,-}$ is non-increasing) in $n$, and,  in addition, the comparison principle gives 
\[ \| \psi^{n,+} - \psi^{n,-} \|_{\infty; \T^d \times [T^{n-1},\infty)} \leq \| \psi^{n,+}(T^{n-1},\cdot) - \psi^{n,-}(T^{n-1},\cdot) \|_{\infty}.\]
The right hand side of  inequality above tends  to $0$, since, if,  for instance, $n$ is odd, 
\[ \psi^{n,+}(T^{n-1},\cdot) - \psi^{n,-}(T^{n-1},\cdot) = \phi^+ - S_{\hat{H}}(D^k_+) \phi^- \to_{n\to \infty} 0. \]
Hence, the sequences $(\psi^{n,+})_{n\in \N}$ and $(\psi^{n,-})_{n\in \N}$ converge (locally uniformly) to a continuous function $\psi$ on $\T^d \times \R $, which is a solution to \eqref{eq:HJglobal} satisfying \eqref{eq:ineqPsi}.
\vs
Assume now that $\tilde{\psi}$ is another function satisfying \eqref{eq:HJglobal} and \eqref{eq:ineqPsi}. Then, using again the comparison principle, we find  $\psi^{n,-}\leq \tilde{\psi} \leq \psi^{n,+}$, and, after letting $n \to \infty$, we obtain that $\tilde{\psi}=\psi$.
\vs
To prove (ii)  consider next  a solution $\psi$ to \eqref{eq:HJglobal}. In view of  \eqref{asn:reg},  the sequence $ \left( \psi(T^n,\cdot) \right)_{n \geq 0}$ is compact, and, therefore, converges, up to a subsequence $n'\to \infty$, to $\phi \in C(\T^d)$. But then, assuming for instance that the $n'$'s are odd, $\psi(T^{n'-1},\cdot)$ converges to $\phi^+=S_{\hat{H}}(\infty)(\phi)$, which, in turn,  implies that $\psi$ must coincide with the solution associated to $\phi^+$ was constructed in (i).

\end{proof}

Given $\pi =(\phi^+,\phi^-) \in \mathcal{P}$ and $\xi$ as in the previous proposition, we denote by  $\psi^{\pi,\xi}$  the unique solution to   \eqref{eq:HJglobal}.  The lemma below provides information on the behavior  of $\psi^{\pi,\xi}$ as $t\to \infty$. 
\begin{lemma} \label{lem:psifwd}
Assume  \eqref{takis90}, \eqref{takis91}, \eqref{asn:Hinf} and let $\xi \in C(\R; \R)$ be such that 
$\limsup_{t \to - \infty} \xi(t) = - \liminf_{t \to - \infty}  \xi(t)=\limsup_{t \to + \infty} \xi(t) = - \liminf_{t \to +\infty}  \xi(t)= +\infty.$
Then, for each \\
$\pi=(\phi^+,\phi^-) \in \Pc$,
\begin{equation}
\limsup_{t \to  \infty} \psi^{\pi,\xi}(\cdot,t) = \phi^+ \ \ \ \text{and} \ \ \   \liminf_{t \to \infty} \psi^{\pi,\xi}(\cdot,t) = \phi^-. 
\end{equation}
\end{lemma}

\begin{proof}
Let $(T_{n})_{n\in \N} $ be an increasing to  $+\infty$ sequence such that, for all $n\in \N$,
\[\xi(T_{2n}) = \min_{[0,T_{2n}]} \xi \ \text{and} \  \xi(T_{2n+1}) = \max_{[0,T_{2n+1}]} \xi, \]
and set  $D_n = \xi(T_{n+1}) - \xi(T_n)$. The $D_n$'s have alternating signs, and $|D_n|$ is an increasing sequence which diverges to $+\infty$.
\vs
Then 
\[ S_H(D_{2n}) \phi^+ = \phi^+ \geq  \psi(T_{2n+1}, \cdot) \geq S_H(D_{2n}) \phi^- \to_{n \to\infty} \phi^+, \]
and, similarly,  $\psi(T_{2n}, \cdot) \to \phi^-$ as $n \to \infty$.
\end{proof}

\begin{remark}
For $t \in \R$, and $\xi$ as in Proposition \ref{prop:psi}, let $\xi^t := \xi(\cdot-t)$. It is clear that, for any $\pi \in \mathcal{P}$,
\[ \psi^{\pi, \xi^t} = \psi^{\pi, \xi} (\cdot,\cdot - t). \]
In particular, if $\xi$ is drawn according to a process whose law is invariant with respect to time-shifts, then the same invariance holds for the law of $\psi$.  This is, for instance, the case  when $\xi$ is a two-sided Brownian motion. 
\end{remark}

The main result of this section asserts that any solution on $ \T^d \times [0,\infty)$ gets, in the $t\to \infty$ limit, close to a global in time solution $\psi$ as in Proposition \ref{prop:psi}. 

\begin{theorem} \label{thm:Hconvfwd}
Let $H$ satisfy \eqref{takis90}, \eqref{takis91} and \eqref{asn:Hinf} and assume that $\xi \in C([0, \infty))$ is such that 
$\limsup_{t\to \infty} \xi(t) = -\liminf_{t\to \infty} \xi(t) = +\infty.$
Then, for any solution $u$ to
$$du + H(x,Du)\circ d\xi(t) = 0\mbox{ on } \T^d \times [0,\infty)  \;\;\; u(0,\cdot) = u_0 \in C(\T^d),$$ 
there exists a conjugate pair $(\phi^+,\phi^-)$ such that
\begin{equation} \label{eq:uFwd}
u(x,t) = c  \;\xi(t) + \psi(t,x) + o_{t\to +\infty}(1)
\end{equation}
where 
$$\phi^- \leq  \psi \leq \phi^+, \ \  \limsup_{t \to \infty} \psi(\cdot,t) = \phi^+ \ \ \text{and} \ \  \liminf_{t \to \infty} \psi(\cdot,t) = \phi^-. $$
\end{theorem}

\begin{proof}
Replacing $H$ by $\hat{H}$ if necessary, we may assume that $c=0$.
\vs

Without loss of generality, we may assume that $\xi$ is extended to $\R$ and is unbounded for negative times, so that there  exists a global solution $\psi$ associated to each conjugate pair as in Proposition \ref{prop:psi}. Of course, $\psi$ may depend on the choice of the extension, but  \eqref{eq:uFwd} will hold for all choices.
\vs

Let $T_{2n}, T_{2n+1}$, $D_n$ be as in the proof of Lemma \ref{lem:psifwd}. Due to the assumed equicontinuity, there exists  a subsequence $n'$ such that  $u(T_{2n'})$ converges to some $\phi^-$. Then by assumption \eqref{asn:Hinf}, $u(T_{2n'+1})$ converges to $\phi^+ = S_H(\infty) \phi^-$.

\vs

Note that Lemma \ref{lem:mon} implies that, for each $n$,
$$S_H(T_{2n+2}) S_H(T_{2n+1})(\phi^-) \geq S_H(D_{2n+2} +D_{2n+1})(\phi^-) = \phi^-$$
and by induction, for all $n,m \geq 0$,
\begin{equation}\label{takis95}
\S_H^{\xi,[T_{2n},T_{2n+2m}]}(\phi^-) \geq \phi^-.
\end{equation}

It follows that, for any other subsequence $n'' = n' +  m'$ with $u(T_{2n''}) \to \hat{\phi}^-$, 
$$\hat{\phi}^- = \lim S_H^{\xi,[T_{2n'},T_{2n'+2m'}]}( u(T_n)) \geq \phi^-$$
and, after reversing the roles of $\phi^-$, $\hat{\phi}^-$, we obtain that in fact \eqref{takis95} is an equality, that is. 
$$\phi^- = \lim_{n \to \infty} u(T_{2n})$$
and then obviously
$$\phi^+ = \lim_{n \to \infty} u(T_{2n+1}) = S_H(+\infty) \phi^-.$$

{  Letting $\psi$ be the solution given by Proposition \ref{prop:psi} and satisfying  $\phi^- \leq \psi \leq\phi^+$, the comparison principle then} implies that $\lim_{n \to \infty} \left\| u(\cdot,T_n)-\psi(\cdot,T_n)\right\|_{\infty}=0$.

\end{proof}

The pair $(\phi^+, \phi^-)$ above is in general not unique, so that the limiting behavior may depend on both the signal $\xi$ and the initial condition $u_0$. Recall that, in the deterministic case ($\xi(t)=t$), $\phi^+$ can be obtained from $u_0$ via a rather explicit Hopf-type formula (see \cite{DS06}). Whether there exists such a simple relation in the case of oscillating $\xi$ is an open question.
\vs

We now turn to \eqref{eq:HJstoc} with  $u_0 \in C(\T^d)$ fixed and $B = B(\omega)$ drawn according to the probability distribution of a (two-sided) Brownian motion, {  and let $\Pi(\omega) = (\Phi^+(\omega), \Phi^-(\omega))$ be the (random) conjugate pair given by Theorem \ref{thm:Hconvfwd}. We now discuss the law of the random variable $\Phi^+$.} Its support in the space $C(\T^d)$ is easily characterized by the following proposition. 
\begin{proposition}\label{prop:support} (i)~We have
\begin{equation}\label{takis96}
\mbox{Supp}(\Phi^+) = \overline{ \left \{ S_{\hat{H}}(\infty) \mathcal{S}_{\hat{H}}^{\xi;[0,1]} u_0 ; \;\xi \in C([0,1]) \right\} },
\end{equation}
where the closure is taken with respect to the  sup-norm.
\vs
(ii)~Let 
$ \phi^+_{+\infty} = S_{\hat{H}}(\infty) u_0$ and $  \phi^+_{-\infty} =  S_{\hat{H}}(\infty)  S_{-\hat{H}}(\infty) u_0.$
Then $\phi^+_{+\infty}$ (resp. $\phi^+_{-\infty}$) is the smallest (resp. largest) function in $\mbox{Supp}(\Phi^+)$.
\vs
(iii)~If $\phi^- \leq u_0 \leq \phi^+$, for some $(\phi^+,\phi^-) \in \mathcal{P}$, then $\mbox{Supp}(\Phi^+) = \{\phi^+\}$.
\end{proposition}

\begin{proof} 

To prove (i) observe that $\mbox{Supp}(\Phi^+)$ is contained in the set in the right, since 
$$\Phi^+(\omega) = S_{\hat{H}}(\infty) \Phi^+(\omega) = \limsup_{t \to \infty} S_{\hat{H}}(\infty) S_{\hat{H}}^{\xi;[0,t]} u_0.$$
\vs
For the reverse inclusion, let $\phi^+ = S_{\hat{H}}(\infty) S^{\xi}_{\hat{H}} u_0$ and define  $\xi^{(n)} \in C([0,2])$  by $\xi^{(n)} = \xi $ on $[0,1]$ and $\dot{\xi}=n$ on $[1,2]$. Then, the  continuity of the solution map yields 
$\ep_n\to 0$ such that, if 
$\| B - \xi^{(n)}\|_{\infty;[0,2]} \leq n^{-1}$, then 
$\|u(2,\cdot) - \phi^+\|_{\infty} \leq \varepsilon_n,$
and, hence, 
\[\|\Phi^+(\omega) - \phi^+\|\leq \ep_n.\]
Since, in view of the full support property of Brownian motion, the event $\{\omega \in \Omega: \; \| B - \xi^{(n)}\|_{\infty;[0,2]} \leq n^{-1}\}$ has positive probability, it follows that $\phi^+$ is in the support of $\Phi^+$.
\vs

To prove (ii), let $\xi \in C([0,1])$ with $\xi(0)=0$. Then Lemma \ref{lem:inf} yields that
\[ S^\xi_{\hat{H}} u_0 \geq  S_{-\hat{H}}(\gamma) S_{\hat{H}}(D) u_0, \]
where $D = \max_{[0,1]} \xi$, $\gamma = \max_{[0,1]} \xi - \xi(1)$ both being non-negative. 
\vs
It then follows from Lemma~\ref{lem:mon} that
\[ S_{\hat{H}}(\infty) S_{-\hat{H}}(\gamma) S_{\hat{H}}(D) u_0 \geq S_{\hat{H}}(\infty) u_0,\]
which implies that  $\phi^+_{+\infty}$ is indeed the smallest element of the support.

{ 
Similarly, for $D' = -\min_{[0,1]} \xi$, $\gamma' = -\min_{[0,1]} \xi + \xi(1)$, we have by Lemma \ref{lem:inf} that
\[ S^\xi_{\hat{H}} u_0 \leq  S_{\hat{H}}(\gamma') S_{-\hat{H}}(D') u_0 \leq S_{} \]
and then by Lemma~\ref{lem:mon}
\[ S_{\hat{H}}(\infty) S^\xi_{\hat{H}} u_0 \leq  S_{\hat{H}}(\infty) S_{-\hat{H}}(D') u_0 \leq  S_{\hat{H}}(\infty) S_{\hat{H}}(n)  S_{-\hat{H}}(n)S_{-\hat{H}}(D') \to_{n \to \infty} S_{\hat{H}}(\infty)  S_{-\hat{H}}(\infty) u_0,  \]
so that $\phi^+_{-\infty}$ is the largest element of the support.
}

\vs

Finally, for (iii) observe that the assumption on $u_0$ implies that $\phi^+_{+\infty} = \phi^+_{-\infty} = \phi^+$, hence,  the result is a consequence of (ii).

\end{proof}

 A simple example, already considered in section \ref{qc}, is the homogeneous Hamiltonian  $H=H(p)$ with $H(0)=0<H(p)$ for $p\neq 0$. Then $c=0$, $\mathcal{E}^+ = \mathcal{E}^- = \R$, and Proposition~\ref{prop:support} implies that 
$$ \mbox{Supp}(\Phi^+) = \left[ \inf u_0, \sup u_0 \right].$$
In particular, in this case $\Phi^+(\omega)$ is not a constant random variable unless $u_0$ is constant.
\vs
 It is also not difficult to construct an example where $\Phi^+(\omega)$ is not even constant modulo additive scalars. Indeed, in view of the previous conclusion, it suffices to construct $H$ and $u_0$ so that the difference of $S_{\hat{H}}(\infty) u_0$ and  $S_{\hat{H}}(\infty) S_{-\hat{H}}(\infty) u_0 $ is not constant. 
\vs 

Next we consider the large-time value of a solution to \eqref{eq:HJstoc1}, still in the case of a driving Brownian motion, where,  in view of Theorem \ref{thm:Hconvfwd}, 
\[ u(x,t) = cB(t) + \psi^{\Pi(\omega)}(x,t,\omega) +o_{t\to +\infty}(1). \]
It follows that $u(\cdot,t)$ depends, up to a small error, on (i)~the macroscopic position given by $c B(t)$, (ii)~the limiting upper/lower profiles $\Pi=(\Phi^+,\Phi^-)$, and (iii)~the actual value of the associated stationary solution $\psi(t,\omega)$ at time. The result below states, roughly speaking, that these three factors are independent in the large time limit.

\begin{theorem}
Assume \eqref{takis90}, \eqref{takis91} and \eqref{asn:Hinf}, and let $u$ be a $\T^d-$periodic solution to \eqref{eq:HJstoc1} with $B$ Brownian motion. Then, there exist three independent random variables $Z \sim \mathcal{N}(0,1)$  and $B,B'$  two-sided Brownian motions, such that, in law and 
with respect the  sup-norm topology for the second component, 
\begin{equation}
 \left( \frac{B(t)}{\sqrt{t}}, u(\cdot,t) - c B(t) \right)\underset{t\to\infty} \to \left(Z,  \psi^{\Pi(B), B'}(\cdot,0) \right).
\end{equation}
\end{theorem}

\begin{proof}
We split $[0,t]$ into the three disjoint intervals 
$$I^1_t= [0,\sqrt{t}], \  I^2_t=[\sqrt{t}, t-\sqrt{t}] \ \text{and} \  I^3_t=[t-\sqrt{t},t],$$ 
noting that $\sqrt{t}$ could be replaced by any function growing to infinity slower than $t$. 
\vs

The idea of the proof is that, up to a small error, each of $\Phi^+$, $\frac{B(t)}{\sqrt{t}}$, $\psi(\cdot,t)$ are determined respectively by the value of the increments of $B$ on these subintervals. 
\vs

Fix $\varepsilon > 0$. It suffices to show that, for each $t >0$ and $i=1,2,3$, there exist independent Brownian motions $B^{i,t}$ such that
\begin{equation} \label{eq:convB123}
\lim_{t \to \infty} \P \left(\frac{ \left|B^{2,t}(t) - B(t)\right|} {\sqrt{t}} + \left\| u(\cdot,t) - c B(t) - \psi^{\Pi(B^{1,t}), B^{3,t}}(\cdot,t) \right\| \leq 3 \varepsilon \right) = 1.
\end{equation}

Let 
\[A_t := \left\{ \text{there exist $\phi \in \mathcal{E}^+$ \; and  \ $s \in I^1_t$ such that  \; $\left\| u(\cdot,s) - \phi \right\| \leq \varepsilon$}\right\}, \]
\vs
and note that there exists some  (deterministic) $T_\varepsilon$ so that 
\[ A_t \supset \left\{\max_{s \in I^1_t} \left( B(s) - \min_{[0,s]} B \right) \geq T_\varepsilon \right\}. \]
Indeed, it suffices to choose  $T_\varepsilon$ such that
\[\sup_{\xi \in C([0,1])} \left\| S_{\hat{H}}(\infty) \S_{\hat{H}}^{\xi,[0,1]}u_0 - S_{\hat{H}}(T_{\varepsilon}) \S_{\hat{H}}^{\xi,[0,1]}u_0 \right\|_{\infty} \; \leq \; \varepsilon , \]
which exists by \eqref{takis91} and \eqref{asn:Hinf}.
\vs
It then follows from the scaling properties of the Brownian motion that, as $t \to \infty$, 
\[ \P(A_t) \geq \P \left( \max_{s \in [0,1]} \left( B(s) - \min_{[0,s]} B \right) \geq t^{-1/4} T_{\varepsilon} \right) \to 1. \] 
\vs
Also note that, for $\phi$ as in the definition of the event, on $A_t$ we have $\| \Phi^+ - \phi \| \leq \varepsilon$.
\vs
Furthermore,  let
\[ A'_t :=  \{ \max_{I^3_t} B - \min_{I^3_t} B \geq T'_\varepsilon \}. \]
It follows that, if $T'_\varepsilon$ is large enough, then, on $A'_t$, 
\begin{equation} \label{eq:onA't}
\sup_{\phi^+ \in \mathcal{E}^+} \sup_{\xi \in C([0,1])} \left\| \S_{\hat{H}}^{B, I^3_t} \S_{\hat{H}}^{\xi,[0,1]} \phi^+ -  \psi^{B,\phi^+} \right\| \leq \varepsilon. 
\end{equation}
Indeed, it suffices to choose $T'_{\varepsilon}$ so that 
\[ \sup_{\phi^+ \in Supp(\Phi^+(u_0))} \|\phi^+ - \S_{\hat{H}}(T'_{\varepsilon})(\phi^-) \|_{\infty} + \|\phi^- -\S_{-\hat{H}}(T'_{\varepsilon})(\phi^+) \|_{\infty} \leq \varepsilon,\]
with $\phi^-$ the conjugate of $\phi^+$.
\vs
Let $\tau_1$ and $\tau_2$ 
be the times where $B$ attains its minimum and maximum in $I^3_t$. Then,  assuming for instance that $\tau_1< \tau_2$, we find that,  for any $\xi$ in $C([0,1])$,
\[ \S_{\hat{H}}^{B, I^3_t} \S_{\hat{H}}^{\xi,[0,1]} \phi^+ \leq \S_{\hat{H}}^{B,[\tau_2, t]} \phi^+ \]
and 
\[ \S_{\hat{H}}^{B, I^3_t} \S_{\hat{H}}^{\xi,[0,1]} \phi^+ \geq \S_{\hat{H}}^{B, [\tau_1,t]} \phi^- \geq  \S_{\hat{H}}^{B,[\tau_2,t]} S_{\hat{H}}(T'_\varepsilon) \phi^- \geq  \S_{\hat{H}}^{B,[\tau_2,t]} \phi^+ - \varepsilon.\]
\vs
Since $\phi^- \leq \psi \leq \phi^+$ at all times, $ \psi^{B,\phi^+}$ also satisfies the same inequalities, \eqref{eq:onA't} follows.
\vs

For $i=1,2,3$,  let $B^{i,t}$ be independent Brownian motions with the same increments as $B$ on $I^i_t$ and note that on $A_t$, again choosing  $\phi$ as in the definition of the event, it holds that
\[  \| \Phi^+(B^{1,t}) - \phi \| \leq \varepsilon.\]
 Then, on $A_t \cap A'_t$, we have 
\begin{align*}
u(\cdot,t) - c B(t) &= S_{\hat{H}}^{B;I^3_t} S_{\hat{H}}^{B; [s,t -\sqrt{t}] } \phi \pm \varepsilon =   \psi^{\phi, B^{3,t}}(\cdot,t) \pm 2 \varepsilon = \psi^{\Phi^+(B^{1,t}), B^{3,t}}(\cdot,t) \pm 3 \varepsilon,
\end{align*}
where by $\pm \varepsilon$ we mean any function with sup-norm less than $\varepsilon$. 
\vs
Then  \eqref{eq:convB123} follows, since we clearly also have that
\[\P \left(\frac{B^{2,t}(t) - B(t)} {\sqrt{t}} \geq \varepsilon \right)\to 0 \ \  \text{as} \ \ t \to \infty. \]
\end{proof}

We remark  that in many cases of interest, solutions of $\mathcal{E}^+$ and $\mathcal{E}^-$ are unique up to additive constants. This is,  for example,  the case if $H(x,p) = \frac{1}{2} |p|^2 - V(x)$ with  $V\in C(\T^d)$ attaining its minimum at a unique point of $\T^d$, or if $H(x,p) = a(x) \sqrt{1+|p|^2}$ where $a > 0$ attains its maximum at a unique point. 
\vs
For the reader's convenience, we state  the  results in this case, where, in particular, any two solutions to the equation to \eqref{eq:HJstoc} on $[0,\infty) \times \T^d$ become asymptotically close, as $t \to \infty$, up  to a constant.

\begin{corollary}
Assume that $H$ satisfies \eqref{takis90}, \eqref{takis91} and \eqref{asn:Hinf}, and, in addition, that 
\[ \Pc =\left\{(\phi_0^+ + k, \phi_0^- + k)\;:\; k\in \R\right\} \mbox{ for some }  \phi_0^+, \phi_0^- \in C(\T^d). \]
 Then, given $\xi \in C(\R)$ which is unbounded from above and below at $\pm\infty$, 
 \vs
(i)~there exists a unique solution $\psi$ to $d \psi + \hat{H}(x,D\psi) \circ d{\xi} = 0$ on $ \T^d\times \R$ such that $\phi_0^- \leq \psi \leq \phi_0^+$,
\vs
(ii)~given any solution $u$ to $d u + H(x,Du) \circ d{\xi} = 0$ on $\T^d\times [0,\infty)$, there exists a constant\\ $k = k(u(\cdot,0), \xi)$ such that 
\[ \lim_{t \to \infty} \left\| u(\cdot,t) -  c \xi(t) - \psi(\cdot,t) - k\right\|_{\infty} = 0,\]
\vs
(iii)~when $ \xi=B(\omega)$ is a (two-sided) Brownian motion, then, in law and the sup-norm in the second argument, 
{ 
$$\left(\frac{B(t)}{\sqrt{t}}, u(\cdot,t) - c B(t) \right) \underset{{t \to +\infty}}\to \left(Z,  {\psi}^{B'}(\cdot,0)+ k(u_0,B) \right),$$

with $Z\sim \mathcal{N}(0,1)$, and the Brownian motions $B,B'$, are independent.}
\end{corollary}

\section{Quantitative estimates for homogeneous SPDE }\label{sec:1d_homogeneous}

\subsection{Convergence due to  stochastic fluctuations}\label{sec:homogeneous}

We fix $\xi \in C([0,\infty))$ and consider the long-time behavior of parabolic-hyperbolic SPDEs of the form 
\begin{equation}\label{takis10}
du=\partial_{x}(F(\partial_{x}u)) dt +H(\partial_{x}u)\circ{d\xi} \ \ \text{in} \ \ \T\times \R_+,
\end{equation}
 including as a special case
\begin{equation}\label{takis11}
du=\partial_{x}(F(\partial_{x}u)) dt +H(\partial_{x}u)\circ{dB} \ \ \text{in} \ \ \T\times \R_+,
\end{equation}
and prove, exploiting only the fluctuating part $H(\partial_{x}u)\circ{d\xi}$, that solutions converge to constants for large times. In particular, the results also apply to problems with $F\equiv 0$, that is, SPDEs like
\[du=H(\partial_{x}u)\circ{d\xi} \ \ \text{in} \ \ \T\times \R_+.\]

The main focus of this section is on the derivation of quantitative results which lead to improved rates of convergence due to stochasticity in the following sections. 
\vs

We begin with a qualitative convergence result (Theorem~\ref{thm:quantitative}), which can be shown under weaker assumptions than the quantitative statements following below. In this qualitative form and for $F\equiv 0$, this result was first obtained in \cite{Li.So2020b}. 
\vs

Throughout the subsection we assume  that 
\begin{equation}\label{takis40} 
{  H=H_1-H_2  \text{ for two convex functions } H_1, H_2 : \R \to \R, \mbox{ and } H(0)=0.}
\end{equation}
In addition, in the following result we will require that
\begin{equation}\label{takis41}
  H(p)>0\ \ \text{ for all} \ \ p\in\R\setminus \{0\}.
\end{equation}

\begin{theorem}\label{thm:quantitative}
Assume \eqref{takis40} and \eqref{takis41}, $F\in C^1(\R)$ is non-decreasing, $u_0 \in C(\T)$, $\xi\in C([0,\infty))$ with $\limsup_{t\to \infty} |\xi(t)| = +\infty$, and let  $u$ be  the viscosity solution  to \eqref{takis10} with initial condition $u_0$. Then, there exists a constant $u_\infty$ such  that
$$\lim_{t\to \infty}\|u(\cdot,t) - u_\infty \|=0.$$

If $\xi=\xi(\omega)$ is an a.s.\ unbounded and continuous martingale, then  $\E[u_\infty] = \int_{\T} u_0$.
\end{theorem}
\begin{proof} 
We first assume that $u_0 \in C^2(\T)$, consider a smooth approximation $\xi^\ep$ of $\xi$  and the unique viscosity solution  $u^\varepsilon$  to
$$u^\varepsilon_t = \partial_{x}(F(u_x^\varepsilon))+H(u_x^\varepsilon) \dot{\xi}^\varepsilon \ \ \text{in} \ \ \T\times \R_+ \quad  u^\ep(0,\cdot) = u_0 \ \ \text{on} \ \ \T,$$ and  recall  that, in view of \cite{LS98},  as $\ep\to 0$ and for all $T>0$,  $u^\varepsilon \to u$ in $C(\T \times [0,T])$.
\vs
In addition, since $H(0)=0$, we also have (see \cite{LS98}), for all $t\in \R_+$,
\begin{equation}\label{takis11}
\|{u^\varepsilon}(\cdot,t)\|\le\|u_0\| \ \ \text{and} \ \   \|u^\varepsilon_x(\cdot,t)\|\le\|u_{0x}\|. 
\end{equation}
 
Next, we observe that 
${v^\varepsilon}={u}_{x}^\varepsilon$ is the entropy solution to 
\[
{v_t^\varepsilon}=\partial_{xx}(F({v^\varepsilon}))+\partial_{x}{H}({v^\varepsilon})\dot{\xi}^\varepsilon \ \ \text{in} \ \ \T\times \R_+ \quad v^\ep(0,\cdot)=u_{0,x} \ \ \text{on} \  \T, 
\] 
and, thus, for any convex and smooth $E:\R\to \R$, we have
\begin{align*}
\dfrac{d}{dt}\int_\T E({v^\varepsilon}) dx & \le -\int_\T E''({v^\varepsilon})F'({v^\varepsilon})({v_{x}^\varepsilon})^{2} dx -\int_\T E''({v^\varepsilon})v^\varepsilon_{x}{H}({v}^\varepsilon)\dot{\xi}^\varepsilon dx\\
 & =-\int_\T E''({v^\varepsilon})F'({v^\varepsilon})({v}_{x}^\varepsilon)^{2} dx\le0.
\end{align*}
By an additional approximation argument, we may assume that $H, H_1$ and $H_2$  in \eqref{takis41}  are smooth, and, hence, using the previous observation with $E=H_1$ and $E=E_2$  yields  that, for  $i=1,2$, 
\begin{equation}\label{eqn:non-decr}
 t \mapsto N^\varepsilon_i(t) := \int    H_i(u_x^\varepsilon(x,t)) dx  \mbox{ is non-increasing}.
\end{equation}
Thus, the map 
$$t \mapsto N^\varepsilon(t) := \int H(u_x^\varepsilon(x,t)) dx = N^\varepsilon_1(t)-N^\varepsilon_2(t)$$
is of bounded variation, and,  in view of \eqref{takis11},  $\|N^\varepsilon\|_{BV}$ is uniformly bounded. It follows that, along subsequences, $N^\varepsilon \rightharpoonup N$ weak-$\star$ in $BV_{loc}(\R_+)$ for some $N\in BV_{loc}(\R_+)$.
\vs 
We also know that, for all $t\in \R_+$,
\begin{equation}\label{eqn:veps_bounds}
  \|v^\ve(\cdot,t)\| \le \|v_0\| \quad\text{and } \quad\|v^\ve(t)\|_{BV} \le \|v_0\|_{BV}. 
\end{equation}
\vs

Then the Aubin-Lions-Simon Lemma (see  Simon \cite[Theorem 5]{S87}) together with the equicontinuity of the $u^\ve$'s imply that $v^\ve \to v=u_x$ in $C_{loc}(\R_+;L^1(\T))$.
\vs

It follows from the   dominated convergence theorem that
\[
\begin{split}
  & N^\varepsilon(t) \to N(t)=\int_\T H(u_x(x,t)) dx  \ \ \  \text{and} \\[1.2mm]
   & N^\varepsilon_1(t) \to N_1(t) = \int_\T    H_1(u_x(x,t)) dx  \ \text{and} \ N^\varepsilon_2(t) \to N_2(t) = \int_\T    H_2(u_x(x,t)) dx, 
\end{split}
\]
and, finally, in view of \eqref{eqn:non-decr} and \eqref{eqn:veps_bounds},  $N_1$ and $N_2$ are both non-increasing and bounded.
\vs

A standard approximation argument  yields 
\begin{equation}\begin{split}\label{eqn:average-dyn}
\int_\T u^\varepsilon(x,t) dx  - \int_\T u^\varepsilon(x,s)  dx
& = \int_s^t \left( \int_\T H(u^\varepsilon_x(x,s)) dx\right) d\xi^\varepsilon_s= \int_s^t N^\varepsilon d\xi^\varepsilon  \\
&= N^\varepsilon(t)\left( \xi^\varepsilon(t) - \xi^\varepsilon(s)\right) - \int_s^t (\xi^\varepsilon(u)-\xi^\varepsilon(s)) dN^\varepsilon(u),
\end{split}\end{equation}
where the right hand side is a Stieltjes integral. 
\vs
In view of the convergences in $\varepsilon$ shown above, we can pass to the limit $\ep \to 0$ in \eqref{eqn:average-dyn}  to obtain that \begin{equation}\begin{split} \label{eqn:average-dyn-2}
\int_\T u(x,t) dx  - \int_\T u(x,s)  dx
&= N(t)\left( \xi(t) - \xi(s)\right) - \int_s^t (\xi(u)-\xi(s)) dN(u).
\end{split}\end{equation}

If $t>0$ and $s\in [0,t] $ are such that $\xi$ achieves its minimum on $[s,t]$ at $s$ and its maximum at $t$, we deduce from \eqref{eqn:average-dyn-2} that
\begin{equation}\begin{split}\label{eqn:N-bound}
N(t)  &= \dfrac{\int_\T u(x,t) dx  - \int_\T u(x,s) dx}{\xi(t)-\xi(s)} + \int_s^t \frac{\xi(u)- \xi(s)}{\xi(t)-\xi(s)} dN(u) \\[1.5mm]
& \leq \dfrac{\max u_0 - \min u_0}{\xi(t)-\xi(s)}  + N_2(s) - N_2(t).
\end{split}\end{equation}

Finally, since $N^i(t)$ is non-increasing and bounded, the limits  
$$N_\infty := \lim_{t\to\infty}N(t) = \lim_{t\to\infty}N^1(t)-\lim_{t\to\infty}N^2(t)$$
exist.
\vs
In view of  the fact that  $\limsup_{t\to \infty} |\xi(t)| = +\infty$, \eqref{eqn:N-bound} implies that there is a sequence $t_n \to \infty$ along which  $N(t_n) \to 0$. 
\vs

Hence, $N_\infty=0$ and 
\begin{equation}\label{eqn:H-conv}
\lim_{t\to \infty}  \int_{\T} H(u_x(t)) dx =0,
\end{equation}
which, in view of \eqref{takis41}, 
implies that
\begin{equation}\label{takis20}
u_x(\cdot,t) \underset{t\to \infty} \to 0 \ \  \text{in measure.}
\end{equation}

\vs
The equicontinuity and equiboundedness of  the family $(u(\cdot,t))_{t\geq 0}$ yields that, along  subsequences $t_n\to \infty$, $u(\cdot, t_n) \to u_\infty$ uniformly  for some bounded and  Lipschitz continuous $u_\infty$. 
\vs

It follows from  \eqref{takis20}  that  $u_{\infty, x}=0$ a.e.\,which implies that $u_\infty$ is a constant.
\vs

Since  \eqref{takis10} admits a comparison principle,  the constant  $u_\infty$  is independent  of the chosen subsequence and thus 
\[\lim_{t\to \infty} \|u(\cdot,t)-u_\infty\|=0.\] 
\vs

In addition, if $\xi=\xi(\omega)$ is given as paths of an a.s.\ unbounded, continuous martingale,  then, taking expectations in \eqref{eqn:average-dyn-2} gives 
\begin{equation*}\begin{split}
\E\int_\T u(x,t) dx  - \int_\T u_0(x) dx = 0,
\end{split}\end{equation*}
and the claim follows after letting $t \to \infty$.

\vs

It remains to remove the assumption that $u_0$ is smooth. Indeed, given
$u_0 \in \blue{C(\T)}$,  choose $u_0^{\pm,\ve} \in C^2(\T)$ such that  
$$\lim_{\ep\to 0}\|u_0^{\pm,\ve} - u_0\|=0, \ \ \|u_0^{+,\ve}-u_0^{-,\ve}\| \le \ve \ \ \text{and} \ \ u_0^{+,\ve} \ge u_0 \ge u_0^{-,\ve}.$$

It follows from the comparison principle for \eqref{takis10} and the asymptotic behavior of the solutions for smooth data that there are constant $u^{\pm,\ve}_\infty$ such that 
$$u^{+,\ve}(\cdot,t) \ge u(\cdot,t) \ge u^{-,\ve}(\cdot,t) \ \ \text{ and, as $t\to \infty$}, \ \  
u^{\pm,\ve}(\cdot,t) \to 
u^{\pm,\ve}_\infty.$$ 

Since $$\|u^{+,\ve}(\cdot,t)-u^{-,\ve}(\cdot,t)\|\le\|u_0^{+,\ve}-u_0^{-,\ve}\|\le \ve,$$
we have   $$|u^{+,\ve}_\infty-u^{-,\ve}_\infty| \le\ve$$ and  the claim follows.

\end{proof}

\begin{remark} In the setting of Theorem \ref{thm:quantitative} with  $\xi(\omega)$ an a.s.\ unbounded, continuous martingale, there is, in general, no additional information on $u_\infty$, which as shown in \cite{Li.So2020b} can be random. However, when  $H(p)=|p|$ and $F\equiv 0$, if $u$ solves \eqref{takis10}, then,  for any non-decreasing and  continuous $\varphi$, $\varphi (u)$ is also a solution, and, hence, 
$$\E[\varphi(u_\infty)] = \int_{\T} \varphi(u_0)dx,$$
a fact which implies that $Law(u_\infty) = (u_0)_*dx$.
\end{remark}
The next result is a quantitative version of Theorem~\ref{thm:quantitative}. For this, we assume that 
\begin{equation}\label{takis43}
\text{ there exists $\alpha \in (0,1]$ such that $H\geq \alpha H_1$,}
\end{equation}
a fact which  can be seen as a way to quantify the growth of the Hamiltonian $H$ away from its minimum $H(0)=0$. 
\vs

\begin{theorem}\label{thm:quantified_difference_convex}
Assume  \eqref{takis40}, \eqref{takis41}, \eqref{takis43}, $F\in C^1(\R)$ is non-decreasing, $u_0 \in \blue{C(\T) \cap W^{1,1}(\T)}$, $\xi\in C([0,\infty))$,  and let  $u$ be the  solution  to \eqref{takis10} with initial condition $u_0$. Then, there exists $C=C(\alpha)>0$ such that, for all $T\ge0$,
\begin{equation}
\displaystyle{\int_{\T} }H_1(u_x(x,T))\,dx \leq C \dfrac{\max u_0 - \min u_0}{\max_{t \in[0,T]} (\xi(t) - \min_{[0,t]} \xi)}.
\end{equation}
If, moreover, $H_1$ is even, then, for all $T\ge0$,
\begin{equation}
\displaystyle{\int_{\T}} H_1(u_x(x,T)) dx \leq C \dfrac{\max u_0 - \min u_0}{\max_{[0,T]}\xi - \min_{[0,T]} \xi}.
\end{equation}
\end{theorem}
\begin{proof}
As in the proof of Theorem~\ref{thm:quantitative} we first assume that $u_0 \in C^2(\T)$. 
\vs

Fix an arbitrary time interval $[s_0,t_0]$ and let $s,t\in [s_0,t_0]$ be such that  $s\leq t$ and 
  $$ \max_{t \in [s_0,t_0]} (\xi(t) - \min_{[s_0,t]} \xi) = \xi(t)-\xi(s).$$
Using an argument similar to the one leading to \eqref{eqn:N-bound}  and recalling that $dN = dN^1 - dN^2$ with $N^1$ 
and $N^2$ non-increasing, we get 
\begin{align*}
N(t)  & \leq \dfrac{\max u_0 - \min u_0}{\xi(t)-\xi(s)}  + N_2(t) - N_2(s)\leq  \frac{\max u_0 - \min u_0}{\xi(t)-\xi(s)}  + N_2(s) - N_2(t).
\end{align*}
Since $N_1$ is non-increasing and,   in view of \eqref{takis43}, $N_2 \leq (1-\alpha) N_1$,  we finally get 
\begin{equation} \label{eq:estN1N2}
N_1(t_0) \leq N_1(t) \leq\frac{\max u_0 - \min u_0}{\max_{t \in [s_0,t_0]} (\xi(t) - \min_{[s_0,t]} \xi)} + (1-\alpha) N_1(s).
\end{equation}
\vs
To obtain the general estimate, we first notice that, without loss of generality, we can choose $\eta>0$ such that $\beta:=1-\alpha+ \eta < 1$ and $N\in\N$ maximal so that
\begin{equation}\label{eqn:bound}
  \max_{t \in[0,T]} (\xi(t) - \min_{[0,t]} \xi) \geq  (\max u_0 - \min u_0) \dfrac{1}{N_1(0) \eta} \sum_{k=1}^N \beta^{-k}. 
\end{equation}
Otherwise, 
\[ N_1(T) \le N_1(0)  \leq  \dfrac{1}{\eta \beta} \frac{\max u_0 - \min u_0}{\max_{t \in[0,T]} (\xi(t) - \min_{[0,t]} \xi)} \]
and nothing remains to be shown.
\vs
Next, notice that we can choose a partition of $[0,T]$ consisting of intervals $I_k=[t_k,t_{k+1}]$, with $k=1,\ldots, N$, such that 
$$\sum_{k=1}^N \max_{t \in [t_k,t_{k+1}]} (\xi(t) - \min_{[t_k,t]} \xi) \ge  \max_{t \in[0,T]} (\xi(t) - \min_{[0,t]} \xi)$$
and, in view of \eqref{eqn:bound}, 
$$\frac{\max u_0 - \min u_0}{\max_{I_k} \xi - \min_{I_k} \xi} = \eta (1-\alpha+\eta)^{k-1} N_1(0).$$

Applying \eqref{eq:estN1N2} sequentially on the intervals $I_k$ we obtain that 
\[N_1(T) \leq  \beta^{N} N_1(0).\]

Since $N$ is chosen maximal for \eqref{eqn:bound} we have, for some constant $C=C(\beta)>0$,
\begin{equation*} 
  \max_{t \in[0,T]} (\xi(t) - \min_{[0,t]} \xi)  \le C \dfrac{\max u_0 - \min u_0} {N_1(0) \eta \beta^{N}}
\end{equation*}
and we conclude that
\[N_1(T) \leq  \beta^{N} N_1(0)\;\leq C \dfrac{\max u_0 - \min u_0}{\max_{t \in[0,T]} (\xi(t) - \min_{[0,t]} \xi)}.\]

To obtain the result for $u_0 \in \blue{C(\T) \cap W^{1,1}(\T)}$, we approximate $u_0$ by smooth functions $u_0^\ve$ obtained by mollification. Since, as $\ep\to 0$,  $u_0^\ve \to u_0$ in $W^{1,1}$, we have that $u^\ve \to u$ uniformly and $u^\ve_x \to u_x$ in $C_t(L^1_x)$. Then the statement follows from convexity of $H_1$ and Fatou's Lemma.
\vs

For the proof of the last claim, we observe that, by symmetry, we also have the estimate
\begin{equation}
\int H_1(- u_x(T)) dx \leq C \dfrac{\max u_0 - \min u_0}{\max_{t \in[0,T]} ( - \xi(t) + \max_{[0,t]} \xi)},
\end{equation} 
which implies the estimate in case $H_1$ is even.
\end{proof}

Next, we discuss a number of concrete examples.

\begin{example}
We consider the $2-$d stochastic mean curvature flow in graph form perturbed  by  homogeneous noise
\begin{equation}\label{eqn:fluc_interface}
du= \delta  \frac{\partial_{xx}u}{1+(\partial_{x}u)^{2}} dt +\sqrt{1+(\partial_{x}u)^{2}}\circ{dB},
\end{equation}
with initial condition  $u_0 \in \blue{C(\T) \cap W^{1,1}(\T)}$, and $B$ Brownian motion. 
A simple calculation shows that \eqref{eqn:fluc_interface} can be recast as \eqref{takis10} with $F(p)=\delta \arctan(p)$ and $H(p)=\sqrt{1+|p|^2}$. 
\vs

The only issue with  applying Theorem~\ref{thm:quantified_difference_convex} is that $H(0)=1$. This can be resoved by an elementary change of unknown: Indeed, it is immediate that  $\tilde u := u-B$ solves
\[d {\tilde u}= \delta  \frac{\partial_{xx} \tilde u}{1+(\partial_{x} \tilde u)^{2}} dt +(\sqrt{1+(\partial_{x}\tilde u)^{2}}-1)\circ{dB}\]
while $\tilde u_x=u_x$. 
Now $H(p)=\sqrt{1+p^{2}}-1$ satisfies all the assumptions of Theorem~\ref{thm:quantified_difference_convex}. 
Since the conclusions of the theorem concern $\tilde u_x$, next we write them using $u_x$.
\vs

It follows that 
\begin{equation}
\int_\T (\sqrt{1+u_x^{2}}-1)dx \leq C \dfrac{\max u_0 - \min u_0}{\max_{t \in[0,T]} B- \min_{[0,T]} B}
\end{equation}
and, thus, 
\[
\|u(T)-\int_{\T}u(T)dx\|_{\infty}\le C \dfrac{\max u_0 - \min u_0}{\max_{t \in[0,T]} B- \min_{[0,T]} B}.
\]
The last estimate is pathwise and quantified and  improves  \cite[Proposition 4.1]{ER12}, which proved the qualitative convergence in mean $\E\|u(T)-\int_{\T}u(T)dx\|_{H^1}^2\to 0$.
\end{example}

\begin{example}
We consider the Ohta-Kawasaki equation with spatially homogeneous noise (see, for example, 
Kawasaki and Ohta \cite{KO82} and Katsoulakis and Kho \cite{KK01}) 
\begin{equation}\label{eqn:OK}
du=\delta\,\frac{\partial_{xx}u}{1+(\partial_{x}u)^{2}}dt +(1+(\partial_{x}u)^{2})^{\frac{1}{4}}\circ{dB}
\end{equation}
with $\delta\ge0$, and Lipschitz continuous initial condition $u_0$, which, as per the discussion in the previous example, can be transformed to 
\[d {\tilde u}= \delta  \frac{\partial_{xx} \tilde u}{1+(\partial_{x} \tilde u)^{2}} dt +(1+(\partial_{x}\tilde u)^2)^{\frac{1}{4}}-1)\circ{dB}\] without changing the $x-$derivative.  
\vs
It is immediate that  $H(p)=(1+p^2)^{\frac{1}{4}}-1$ satisfies \eqref{takis40}. 
\vs
Moreover, 
\begin{align*}
  H'(p) &= \frac{1}{2}(1+p^{2})^{-\frac{3}{4}}p, \quad
  H''(p) = \frac{1}{2}(1+p^{2})^{-\frac{3}{4}}-\frac{3}{4}(1+p^{2})^{-\frac{7}{4}}p^2 =: H_1''-H_2'',
\end{align*}
with
\begin{align*}
 H_1(p):=\int_0^p \int_0^{p_1} \frac{1}{2}(1+p_2^{2})^{-\frac{3}{4}}\,dp_2\,dp_1
\end{align*} 
and, since $H_1'' \ge 0$ and $H_2'' \ge 0$ , it follows  that $H=H_1-H_2$ is the difference of two convex functions. 
\vs
Finally, to check \eqref{takis43} we  observe that, for  $|p|\le K$, there exists  $\eta\in(0,1)$ so that $H''\ge\eta H_1''$. Moreover, there is a constant $C$ such  that $$H(p)\ge C \dfrac{1}{1+K^{3/2}} |p|^2.$$
Since $\|u_x(\cdot,t)\|\le\|u_{0,x}\|$, 
it follows that we can thus apply Theorem \ref{thm:quantified_difference_convex}.
\vs
Hence, there is a constant $C>0$, which may depend on $\|u_0\|_{Lip}$ but not on $\delta$ such  that 
\begin{equation}
\|u(\cdot,T)-\int_{\T}u(x, T)dx\|_{\infty} \leq  \frac{C}{\osc_{0,T}B}.
\end{equation}
\end{example}

\begin{remark}
As seen in the two examples with $\delta=0$, in contrast to  \cite[Proposition 4.1]{ER12}, the proof given here does not rely on the (degenerate) viscosity of the mean curvature operator, but exploits exclusively the convexity of the Hamiltonian in the stochastic term. 
\end{remark}
\vs

Next we show that iterating the estimate in Theorem~\ref{thm:quantified_difference_convex} allows to improve the dependency on the initial condition and the speed of convergence in $T$. 
\vs
In preparation, given a path $\xi : [0,T] \to \R$ and $C>0$, we define the times  $\tau^C_k$ by 
\begin{equation} \label{eq:defTauC}
\tau^C_0 := 0, \;\; \; \tau^C_{k+1} := \inf\left\{ t \geq \tau_k^C, \;\;\; \osc(\xi)_{\tau_k^C,t}=C\right\} \wedge T, 
\end{equation}
and 
\begin{equation} \label{eq:defNC}
  N^C_{0,T}(\xi) := \sup \{k  \geq 0: \tau^C_k < T\},
\end{equation}
 where $\osc(\xi)_{s,t}$ denotes the oscillation of $\xi$ on the time interval $[s,t]$. 
\vs
We have the following auxiliary lemma.

\begin{lemma}\label{lem:lln}
   Let $B$ be a Brownian motion. Then $N^C_{0,T}(B)\underset{T\to \infty} \sim c C^{-1/2}  T$ almost surely, where $c =\E[\tau_1^1] < \infty$.
\end{lemma}
\begin{proof}
Note that by the strong Markov property, $(\tau^C_{n+1} - \tau^C_{n})_{n \geq 0}$ are i.i.d., and, in view of the Brownian scaling,   equal in law to $C^{1/2} \tau^1_1$. Then the strong law of large numbers implies that, a.s., 
\[ \tau^C_N \sim_{N \to \infty} N C^{1/2} \E[ \tau^1_1], \]
which yields the result.
\end{proof}

Next we state the following improvement of Theorem~\ref{thm:quantified_difference_convex} for which we need to quantify the growth of $H$ by assuming  that
\begin{equation}\label{eq:Hp}
\text{there exist  $C_1>0$, $C_2\geq 0$ 
and $q \geq1$ such  that, for all $p\in\R$,}\ \ 
 H_1(p) \geq C_1 \left( |p|^q - C_2 \right). 
\end{equation}

{ 
Recall that the total variation of a function $u :\T \to \R$ is defined by
\[ \|u\|_{{TV}} = \sup \left\{ \int_{\T} u \phi_x, \;\; \phi \in C^1(\T), \| \phi\|_{\infty} \leq 1\right\}, \]
and that it coincides with $\|u_x \|_{L^1}$ when $u \in W^{1,1}(\T)$.
}
%

\begin{theorem} \label{thm:quantifiedbetter} 
Assume  \eqref{takis40}, \eqref{takis43},  \eqref{eq:Hp}, $F\in C^1(\R)$ non-decreasing, $u_0 \in \blue{C(\T)}$, and $\xi\in C([0,\infty))$ with $\limsup_{t\to \infty} |\xi(t)| = +\infty$,  and let  $u$ be the  solution  to \eqref{takis10} 
with initial condition $u_0$. There exist $c,C>0$ depending only on $q$ and $\alpha$ such that, if $q=1$, then, for all $T>0$,  
\begin{equation} \label{eq:oscuT1}
{ \|u(\cdot, T)\|_{{TV}}} \leq e^{-c  N^{C/ C_1}_{0,T}(\xi)} \osc(u_0) +  C C_2,
\end{equation}
and, if $q>1$, 
\begin{equation} \label{eq:oscuT2}
\|u_x(\cdot, T)\|_{L^q}  \leq C \left( \left(C_1 \osc_{0,T}\xi \right)^{-\frac{1}{q-1}} + C_2^{1/q} \right).
\end{equation}
Notably, the latter bound is uniform with respect to the initial condition.
\end{theorem}
\begin{proof}
\blue{As in the proof of Theorem \ref{thm:quantified_difference_convex}, by approximation, it is sufficient to consider $u_0 \in C^2(\T)$.}

Note that \eqref{eq:Hp} and H\"older's inequality imply that
\begin{equation} \label{eq:Hosc}
\osc(u(\cdot,t))^q \leq \|u_x(\cdot,t)\|_{L^q}^q \leq   \frac{1}{C_1} \left( \int H_1(u_x(y,t)) dy\right) \wedge \left( \int H_1(-u_x(y,t)) dy\right) + C_2.
\end{equation}

When $q=1$,  we consider the stopping  times $\tau_k=\tau_k^{2C/C_1'}$ with  $C$  the constant from Theorem \ref{thm:quantified_difference_convex}, which we can apply repeatedly to get that
\[ \osc(u(\cdot,\tau_{k+1})) \leq \|u_x(\cdot,\tau_{k+1})\|_{L^1}  \leq  \frac{\osc(u(\tau_k))}{2} +  C_2,\]
and,  by induction,  
\[\osc(u(\cdot,\tau_{N}))   \leq  \frac{\osc(u_0)}{2^N} +  2C_2 .\]

When $q>1$, we let 
\[ C'' = 2^{q +1} \frac{C}{C_1}, \]
where $C$ is the constant from Theorem \ref{thm:quantified_difference_convex}, and assume that, for some $N \in \Z$, such that 
$ 2^N \geq C_2^{1/q},$
\[ \osc_{0,T} \xi \geq C'' \sum_{k=N}^\infty (2^k )^{1-q} .\]
We can then find  $(t_{k})_{k \leq -N}$ such that 
\[ \lim_{k\to -\infty} t_k=0, \ \ t_{-N} = T\ \ \text{ and } \ \ 
\ \osc_{t_k,t_{k+1}} \xi \geq  C'' (2^k )^{1-q}. \]
Next we note that, if we assume that, for some $k >N$, 
\begin{equation} \label{eq:tk1}
 \osc(u(t_{-k})) \leq   2^k + C_2^{1/q}, 
\end{equation}
then combining again Theorem \ref{thm:quantified_difference_convex} and \eqref{eq:Hosc}, we find
\begin{align*}
\osc(u(\cdot,t_{-k+1})) \leq  \|u(\cdot,t_{-k+1})\|_{L^q}& \leq  \left(\frac{C}{C_1}  \frac{2^k +  C_2^{1/q}}{ C'' (2^k )^{1-q} } + C_2 \right)^{1/q}   \\
& \leq  \left( \frac{2^k+2^N}{2^{q+1} 2^{k(1-q)}} + C_2 \right)^{1/q} \leq 2^{k-1} + C_2^{1/q}.
\end{align*}
 Since it is clear that \eqref{eq:tk1} holds for $k$ large enough,  using  induction we get
 \[ \osc(u(T)) \leq 2^{N} + C_2^{1/q},\]
and choosing  the smallest $N$ satisfying  the required constraints, we conclude.
\end{proof}

\begin{remark}
When  $q>1$, the decay rate in Theorem \ref{thm:quantifiedbetter} is optimal. Indeed, recall that in the first-order deterministic setting, that is,  $F\equiv 0$, $H$ convex, and $\xi(t)=t$, the solution is given by the Lax-Oleinik formula
\[u(x,t) = \underset{y\in \T^1} \sup \left \{ u_0(y) - t L\left(\frac{y-x}{t} \right) \right\}, \]
where $L$ is the convex conjugate of $H$. In particular, if $H(p)=|p|^q$, with $q>1$, then $L(v) = C_q |v|^{q/(q-1)}$. Simple choices of $u_0$ then lead to solutions satisfying $\osc(u(t)) \gtrsim t^{-1/(q-1)}$, which is the same rate as in the theorem.
\end{remark}

\begin{remark} \label{rem:Constants}
In order to get optimal constants in the exponential in \eqref{eq:oscuT1}, we need to replace $2$ by any $\rho>1$ in the proof of Theorem \ref{thm:quantifiedbetter}. This leads to $c = \ln(\rho)$ and $C = \rho C/C_1$. If $\xi=B$ is a Brownian motion,  the overall dependence in the exponential is of order $\ln(\rho) \rho^{-1/2}$, which achieves its maximum value $2/e$ when $\rho=e^2$. 
\end{remark}

We conclude this subsection remarking  that, in the special case of a quadratic Hamiltonian, under a specific condition on the dissipation and when $\xi=B$ is a Brownian motion, we can apply the results of Gassiat and Gess \cite{GG16} to obtain estimates on the Lipschitz constant of the solution. The decay is of the same order in $t$ as that given by Theorem \ref{thm:quantifiedbetter} but in a stronger topology.

\begin{proposition}[Quadratic Hamiltonian]\label{prop:quadratic_H}
Let $u$ solve 
\[ \partial_t u = g(\partial_x u) \partial_{xx} u + \frac{1}{2}(\partial_x u)^2 \circ dB(t) \]
where $B$ is a Brownian motion and $g \geq 0$ is continuous and satisfies, for some $C_g\in [0,1/2)$ and in the sense of distributions,     $g'' \leq C_g.$ 
Then, there exist $(0,\infty)$-valued random variables $(C(t))_{t\geq 0} $ and $Z$ such that,  for each $t>0$,
$C(t) \overset{d} = t^{1/2} Z,$
and
\begin{equation} \label{eq:bndGG} \|u_x(\cdot,t) \| \leq  \frac{1}{ C(t)}.  
\end{equation}
\end{proposition}

\begin{proof}
It follows from  \cite{GG16} that, for each $t>0$, 
\[ \| \left( u_{xx}(\cdot,t) \right)_+ \| \leq \frac{1}{X_t}, \]
where $X$ is the (maximal continuous) solution to
\[ X_0=0, \;\;\; X \geq 0, \;\;\; dX = - \frac{C_g}{X} dt  - dB \mbox{ when } X>0.\]
The fact that $C_g < \frac{1}{2}$ ensures that $X$ is a Bessel process of dimension  in $(0,1)$, and, in particular,  is a.s. positive for a.e. $t>0$.
\vs
Since  $ \|\partial_x u(t) \|$ is non-increasing in $t$ and, in view of the $1-$periodicity of $u$, $\|u_{x}(\cdot,t)\|\leq \|u_{xx}(\cdot,t)_+ \|$, we obtain \eqref{eq:bndGG} with
\[ C(t) = \max_{0 \leq s \leq t}  X_s. \]
The existence of the random variable $Z$ follows immediately from the scaling invariance of Bessel processes.
\end{proof}

\subsection{Convergence due to dissipation}\label{sec:dissipation}

Here we provide quantitative estimates on the long-time behavior of  parabolic-hyperbolic SPDEs of the form 
\begin{equation}\label{eq:ph-spde}
dv=\Div(F'(v)D v) dt+ \sum_{i,j}\partial_{x_{i}}(H^{i,j}(x,v))\circ{d\xi^{j}}\ \ \text{in } \ \ \T^{d}\times \R_+
\end{equation}
dictated by the dissipation due to the parabolic part, that is, the estimates  are uniform in $H$ and $\xi$. In particular, the results apply to the case that $\xi=B(\omega)$ is given by paths of Brownian motion and yield uniform estimates in $\omega$. \blue{In this section, we will always assume that $F\in C^1_{loc}(\R)$ with $F'\ge0$, $H\in C^0(\T^d\times\R;\R^{d\times m})$, $v_0 \in L^1(\T^d)$, and $\xi \in C^0(\R_+;\R^m)$}.
\vs

The key step is  the observation that, given that the noise part of  \eqref{eq:ph-spde} is divergence free,   the usual $L^p-$and entropy-entropy dissipation inequalities known in the deterministic setting can be recovered in a pathwise manner for \eqref{eq:ph-spde}.  As a consequence, it is possible to obtain quantitative estimates for the rate of convergence, that resemble the deterministic setting.

\vs
The well-posedness of SPDEs like  \eqref{eq:ph-spde} is an intricate subject and is known under various assumptions on the coefficients $F,H$ and $B$, see \cite{Li.Pe.So2013,Li.Pe.So2014,Ge.So2015,Fe.Ge2019,Fe.Ge2021,So.19}. Since here we are not focusing on the concept of solutions and their uniqueness, but on uniform apriori estimates, we define solutions as limits of smooth approximations, for which uniform estimates will be shown.

\begin{definition}\label{def:sol}
 A function $v\in C([0,\infty),L^1(\T^d))$ is a solution to \eqref{eq:ph-spde} if there exist sequences $v_0^\ve \in C^2(\T^d)$ with $\int_{\T^d} v^\ve_0 dx=0$, $\xi^\ve\in C^2(\R_+;\R^m)$, $H^\ve \in C^3(\T^d\times\R;\R^{d\times m})$ with $\Div_x H^\ve(x,\cdot)=0$,  $F^\ve\in C^3(\R;\R_+)$ such that $(F^\ve)'\ge F' \vee \ve > 0$ and classical solutions $v^\ve$ to \eqref{eq:ph-spde} with $(v_0, \xi, H, F)$ replaced by $(v_0^\ve, \xi^\ve, H^\ve, F^\ve)$ such that 
   \begin{align*}
   (v_0^\ve, \xi^\ve, H^\ve, F^\ve)\to (v_0, \xi, H, F) 
   \end{align*}
   in $L^1(\T^d)\times C^0(\R_+;\R^m)\times C^0(\T^d\times\R;\R^{d\times m}) \times C^0(\R)$
    and $v^\ve \to v$ in $C([0,\infty),L^1(\T^d))$.
\end{definition}



\begin{theorem}\label{takis60} Fix $v_0 \in L^1(\T^d)$ and let $v$ be a solution of the parabolic-hyperbolic PDE \eqref{eq:ph-spde} with initial value $v_0$. Let $p\in[1,\frac{2d}{d-2}1_{d\ge3}+\infty1_{d\le2})$ and  $  E:\R\to \R_+$ be such that ${ E}''=1/F'$ and ${ E}'(0)= E(0)=0$.  

Then, there exists  $c>0$ depending only on the embedding constant of $H^1 \hookrightarrow L^p$ and not on $H$, $B$ and $F$ as long as $F' \ge 0$ and $1/F' \in L^1_{loc}$, such that,  path-by-path and for all $t>0$, 
\begin{equation}\label{eq:dissipation}
  \|v(\cdot, t)\|_{L^{p}}^{2}\le c \dfrac{\int  E(v(x,0))dx}{t}.
\end{equation}
\end{theorem}

\begin{proof} We prove the estimate for the regular approximations $v^\ve$ from Definition \ref{def:sol} with constant $c$ in \eqref{eq:dissipation} uniform in $\ve$ and noting that $E^\ve\le E$ since $(F^\ve)'\ge F'$. Passing to the limit $\ve\to0$ then provides the claim. For simplicity, we drop the $\ve$'s in the notation in the following.
\vs

Given convex and  smooth ${\overline E}:\R\to \R$ we find 
\begin{align*}
\dfrac{d}{dt}\int_{\T^d} {\overline E}(v(x,t)) dx 
& =-\int {\overline E}''(v)F'(v)|D v|^{2}-\sum_{i}\int {\overline E}''(v)v_{x_{i}}H^{i,j}(x,v)\dot{B^{j}}\\
 & =-\int {\overline E}''(v)F'(v)|D v|^{2}-\sum_{i} \left(\int[{\overline E}''(v)H^{i,j}(x,v)]_{x_{i}}-[{\overline E}''(v)H_{x_{i}}^{i,j}(x,v)]\right) \dot{B^{j}}\\
 & =-\int {\overline E}''(v)F'(v)|D v|^{2} =-\int | D [({{\overline E}''} F')^{1/2}](v) |^2, 
\end{align*}
where the third term vanishes since $H$ is divergence free, and $[\cdot]$ denotes an anti-derivative.
\vs

We now choose ${\overline E} = E$. Then $\int_0^v (\overline E''(s) F'(s))^\frac{1}{2}  ds=v$.
\vs

Since $v$ has mean zero,  Poincar\'{e}'s  inequality and the Sobolev embedding yield, for  $p\in[1,\frac{2d}{d-2}1_{d\ge3}+\infty1_{d\le2})$, and for some $c_P>0$ that  
\begin{align*}
\dfrac{d}{dt}\int_{\T^d}  E(v(x,t)) dx & \le-c_P\|v\|_{L^{p}}^{2} ,
\end{align*}
and, hence, for some constant $c$ depending on the embedding constant $c_P$ only,
\begin{align*}
\int_{\T^d} E(v(x,t)) dx +c\int_{0}^{t} & \|v(\cdot,r)\|_{L^{p}}^{2}dr\le\int _{\T^d}  E(v(x,0)) dx.
\end{align*}

Next, choosing $\overline E(v)=|v|^p$, which can be justified by an approximation argument,  we get that 
$
t\mapsto\|v\|_{L^{p}}^{2}
$
is non-increasing for all $p\in[1,\infty)$.  
\vs
Hence,
\[
\int_{0}^{t}\|v(\cdot,r)\|_{L^{p}}^{2}dr\ge t\|v(\cdot,t)\|_{L^{p}}^{2}
\]
and,  thus, using that $  E\ge 0$, we find
\begin{align*}
 & \|v(\cdot,t)\|_{L^{p}}^{2}\le \dfrac{c}{t} \int_{\T^d}  E(v(x,0)) dx.
\end{align*}
\end{proof}

\begin{remark}
In the case of spatially homogeneous noise, instead of embedding into $L^{p}$, we could embed into homogeneous $\dot{W}^{s,1}$ spaces for $s\in[0,1)$. Due to the $L^{1}-$contraction these homogeneous semi-norms do not increase. This would lead to estimates like
\begin{align*}
 & \|v(\cdot,t)\|_{\dot{W}^{s,1}}^{2}\le  \dfrac{c}{t} \int_{\T^d}  E(v(x,0)) dx.
\end{align*}
\end{remark}
\vs
We discuss next a concrete example. 
\vs
\begin{example} We consider {\it the porous medium and fast diffusion} SPDE
\[
dv=\Delta v^{[m]} dt +\sum_{i,j}\partial_{x_{i}}(H^{i,j}(x,v))\circ{d B^{j}}\ \ \text{in }\ \ \T^{d} \times \R, 
\]
with $m\in(0,2)$, $B\in C^0(\R_+;\R^d)$, $H\in C^0(\T^d\times\R;\R^{d\times d})$, and $Dv^{[m]}:=m |v|^{m-1} Dv$,
which is like \eqref{eq:ph-spde} with $F'(r)=m |r|^{m-1}$. Assume that $v$ is a solution in the sense of Definition \ref{def:sol}. For  $E(r):=|r|^{3-m}$, so that $(E''(v)F'(v))^{\frac{1}{2}}=1$ and $p\in[1,\frac{2d}{d-2}1_{d\ge3}+\infty1_{d\le2})$, we find 
\begin{align*}
  \|v(\cdot,t)\|_{L^{p}}^{2}\le\dfrac{\|v(\cdot,0)\|_{L^{3-m}}^{3-m}}{t}.
\end{align*}
\end{example}

\vs

The next two results are about \eqref{eq:ph-spde} for  $d=1$ with homogeneous Hamiltonians. 
\vs 

When  $F$ is uniformly elliptic, that is, $F'\ge c>0$, we have the usual exponential decay in $L^2-$norm for \eqref{eq:ph-spde}, which we record here for convenience.

\begin{proposition}\label{prop:Gronwall}
Assume that $F' \geq c >0$, and $H$ is homogeneous in $x$ and let $\lambda_1$ be the first eigenvalue of $(-\partial_{xx})$ on $\T$ and \blue{$v=u_x$} a solution of \eqref{eq:ph-spde} in $\T\times \R_+$. 
Then, for all $t \geq 0$, 
\begin{equation*}
\left\|u_x(\cdot,t)\right\|_{L^2} \leq \left\|u_x(\cdot,0)\right\|_{L^2} e^{- \frac{c}{\lambda_1} t}.
\end{equation*}
\end{proposition}
\begin{proof}
Using an approximation argument, we may assume that  $B\in C^1$. It then follows, using Poincar\'e's inequality,  that  
\[ \frac 1 2  \frac{d}{dt}  \int u_x^2(\cdot,t)dx\; \leq  \;  \int u_x F(u_x)_{xx} + \int u_x H(u_x)_x \dot{B} \; =  - \int (u_{xx})^2 F'(u_x) \leq - \frac{c}{\lambda_1} \int (u_x)^2.  \]
We conclude applying   Gronwall's inequality.
\end{proof}

The next proposition is about the degenerate setting,  that is, for $F'\geq 0$. 

\begin{proposition}\label{prop:deg_diss} Assume that $H(x,u)=H(u)$ and  $F'\ge 0$, and let $E:\R\to \R_+$ be smooth and convex,  $G(r)=\int_0^r (F'E'')^\frac{1}{2}(u)du$, and \blue{$v=u_x$} a solution to  \eqref{eq:ph-spde} in $\T\times \R_+$. Then, there exists $C>0$ such that, for all $t \geq 0,$
$$G^2( \|u_x(\cdot, T) \|_\infty)\le \frac{C}{T}\int_\T E(u_x(0)) \,dx.$$
\end{proposition}
\begin{proof}
Arguing as in the proof of Theorem~\ref{takis60} we find, using that $E$ is smooth, convex and $F$ non-decreasing,
\begin{align*}
  \dfrac{d}{dt} \int_\T E(u_x(x,t)) \,dx
  &= -\int_\T  E''(u_x)F'(u_x)|\partial_{x}u_x|^2 \,dx
  \le -\int  E''(u_x)F'(u_x)|\partial_{x}u_x|^2 \,dx \\
  &= -\int  | \partial_{x}G(u_x)|^2 \,dx \le - c  \|G^2(u_x) \|_\infty.
\end{align*} 
Hence, 
\begin{align*}
  \int_\T E(u_x(x,t)) \,dx+c  \int_0^t G^2(\|u_x(\cdot,r)\|_\infty) \, dr
  &\le \int_\T E(u_x(x,0)) \,dx.
\end{align*} 
Since $t \mapsto \|u_x(\cdot,t)\|_\infty$ is non-increasing and $G$ is non-decreasing, we obtain that
\begin{align*}
    c t \|G^2(u_x(\cdot,t)) \|_\infty  &\le \int_\T E(u_x(x,0)) \,dx.
\end{align*} 
\end{proof}

\subsection{Improved bounds and convergence due to the interaction of stochastic fluctuations and dissipation}\label{sec:interplay}

Taking advantage of the interplay of the dissipative behavior of the parabolic part and the averaging behavior of the hyperbolic part, in this section we derive new quantitative estimates for the long-time behavior of space-periodic solutions to parabolic-hyperbolic SPDEs when $d=1$. In particular, these estimates unveil improved decay compared to the deterministic equations.

\begin{theorem} \label{thm:G}
Assume $H\in C^2_{\blue{loc}}(\R)$ is convex and even, $F:\R\to \R$ is odd with $F'>0$ and let $G(r) = \int_0^r (F' H'')^{{ 1/2}}(u) du$, \blue{$u_0 \in C(\T)$}. If $u$ is a periodic in space solution to \eqref{eq:ph-spde} in $\T\times \R_+$, then, for every $T>0$,  
\begin{equation}
\left\| G(u_x(\cdot, T)) \right\|_{\infty}^{ 2} \leq \frac{\osc(u(\cdot, 0))}{\Gamma(B)_{0,T}}, 
\end{equation}
where, for $t(s) = \inf \{ t > s\geq 0:  B(t)< B(s)\}$, 
\begin{equation}\label{takis62}
\Gamma(B)_{S,T} = \sup_{S \leq s \leq T}  \int_{s}^{t(s) \wedge T} (B(u) - B(s)) du.
\end{equation}
\end{theorem}
\begin{proof}
Fix $s$ and $t=t(s) \wedge T$. Using  \eqref{eqn:average-dyn} and the fact that the maps $t\to \sup u(\cdot,t)$ and $t\to - \inf u(\cdot,t)$ are non-increasing, we find 
\begin{align*}
\osc(u(\cdot,0)) &= \sup u(\cdot,0) - \inf u(\cdot,0) \geq \sup u(\cdot,t) - \inf u(\cdot,s) 
\geq \int u(\cdot,t) - \int u(\cdot,s) \\
&= N(t)\left( B(t) - B(s)\right) - \int_s^t (B(u)-B(s)) dN(u) 
\geq - \int_s^t (B(u)-B(s)) \frac{dN(u)}{du} d u
\end{align*}
were  $N(r) = \int_\T H(u_x(\cdot,r))dx $ is non-increasing. 
\vs
In addition,  we have 
\[
\frac{dN(r)}{dr} \leq - \int_\T (F' H'')(u_x(x,r)) (u_{xx}(x,r))^2 dx = - \int_\T (G(u_x(x,r)))_x^2 dx \leq - \|G(u_x(\cdot,r))\|_{\infty}^{ 2}. 
\]

Since $G$ is non-decreasing and $\|u_x(\cdot,t)\|_{\infty}$ is non-increasing in $t$, it follows  that $\|G(u_x(\cdot,t))\|_{\infty}$ is non-increasing, and the proof is  complete.
\end{proof}

Next, we state and prove a result about the long-time behavior of the quantity defined in \eqref{takis62}. In the claim, $f(t)=\text{o}(t)$ means $\lim_{t\to\infty} f(t)/t=0.$

\begin{lemma}\label{lem:T}
  Fix $K>0$, and, for any $t\geq 0$, let $T(t) := \inf \{s >t, \; \Gamma(B)_{t,s} \geq K\}$, where $\Gamma(B)_{t,s}$ is defined by \eqref{takis62}. Then, $\P$-a.s. and as $t \to \infty$, 
  \begin{equation} \label{eq:Tt}
    T(t) = t + \text{o}(t). 
  \end{equation}
\end{lemma}
\begin{proof}
  Note that, since $\Gamma(B)_{t,s} \geq \Gamma(B)_{n,s}$ if $t\leq n \leq s$, it suffices to show the claim for integer $n$'s. In addition, the full support property of Brownian motion implies that, for each $\varepsilon >0$, $p_\ep:=\P(\Gamma(B)_{0,\varepsilon} \geq C)  \in (0,1)$ and we have that
    \[ \P( T(n) - n > \varepsilon n ) =  \P\left( \Gamma(B)_{0,\varepsilon n} \leq C \right)  \leq \P\left( \Gamma(B)_{0,\varepsilon} \leq C\ , \ldots,  \Gamma(B)_{(n-1)\varepsilon,n \varepsilon} \leq C\right) = p_{\varepsilon}^n.\]
It follows from the Borel-Cantelli lemma that, $\P$-a.s.,
    \[  \lim_n \frac{T(n)-n}{n} = 0,\]
and the proof is complete.
\end{proof}

We now show how all the previous results can be combined for  the stochastic mean curvature SPDE
\begin{equation}\label{takis65}
du = \dfrac{\partial_{xx}u}{1+(\partial_x u)^2}\; dt  +\sqrt{1+(\partial_x u)^2} \circ d{B} \ \ \text{in} \ \ \T\times \R_+,
\end{equation} 
and, thereby, that the interplay of stochastic fluctuation with deterministic decay leads to an improved decay of oscillations (see Remark \ref{rmk:mcf_limit}).

\begin{theorem}\label{takis68}
Fix $\delta>0$ and assume that $B=(B(\omega,t))_{t\geq0}$ is  a Brownian motion on a probability space $(\Omega, \mathcal{F},\P)$. 
There exists $c >0$ and a $\P$-a.s. finite random variable $C(\omega)$, such that,  if $u$ is the solution to \eqref{takis65} with initial condition $u_0 \blue{\in C(\T)}$, then, for all $t\geq 0$,  
\begin{equation} \label{eq:OscSmcf}
\osc(u(\cdot,t)) \leq C(\omega) \osc(u_0) e^{-c t}.
\end{equation}
Let $\tau(u_0,B) := \inf \left\{t \geq 0: \;\; \osc(u(\cdot, t)) \leq 2 \right\}$. 
Then there is a deterministic $c>0$, and, for almost every Brownian path, constants  $C(B),  \rho(B)$ such that
\begin{equation} \label{eq:tau}
\tau(u_0,B) \le c \log(1+\osc(u_0)) + C(B) 
\end{equation}
and
\begin{equation} \label{eq:uxSmcf}
\|u_x(\cdot, t)\|_{\infty} \leq C(\omega) \osc(u_0) e^{-c t}  \ \ \text{for all} \ \  t \geq \tau(u_0,B) + \rho(B).
\end{equation}
Finally, for each fixed $u_0$,
\begin{equation} \label{eq:expdecay}
 { 
\limsup_{t \to \infty} \frac{\ln(\|u_x(\cdot,t)\|_{L^2}) }{t} \leq - \frac{\delta}{\lambda_1}.}
\end{equation}
\end{theorem}

\begin{proof}
We first remark that \eqref{eq:tau} is an immediate consequence of \eqref{eq:OscSmcf}.
\vs
For the rest of the claims similarly to Example~1 we introduce   $\tilde u = u-B$ and note  that $\tilde u$ satisfies 
\begin{equation}\label{eqn:spde-quant}
d \tilde u=\partial_{x}(F(\partial_{x} \tilde u)) dt +H(\partial_{x} \tilde u)\circ{dB}\ \ \text{in } \ \ \T \times \R_+ \ \ \text{and} \ \ \tilde u(0, \cdot )=u_0 \ \ \text{on} \ \ \T,
\end{equation}
for $F(\rho)=\arctan (\rho)$ and $H(p) = \sqrt{1+p^{2}}-1$ with  $H(p) \geq |p| -1$.
\vs

The claimed estimates are obtained in the following three steps: (i)~Theorem ~\ref{thm:quantifiedbetter} yields that  $\osc(u(\cdot, t))$ is of  order $1$ after a time which depends logarithmically on $\osc(u_0)$, (ii)~in view of Theorem~ \ref{thm:G}, after the  time given by (i), $\|u_x\|_{\infty}$ is  also of order $1$, and (iii)~the uniform bound on $ u_x$, which is a consequence of (ii) is then used  to obtain exponential convergence to a constant. 
\vs

Here we show (i). Theorem~\ref{thm:quantifiedbetter} implies that, for all $t\geq 0$, 
\[  \osc(u(\cdot,t)) \leq \osc(u_0) e^{-c N^C_{0,t}(B)} +2.\]
Also note that, due to Lemma \ref{lem:lln}, a.s. and  as $t \to \infty$, $N^C_{0,t}(B) \sim C^{-1/2} t$. 
\vs
It follows that,  for some $c>0$, $C(B)$ a.s. finite and for all $t\geq 0$, 
\begin{equation} \label{eq:Osc2}
osc(u(\cdot,t)) \leq C(B) \osc(u_0) e^{-c t} +2. 
 \end{equation}
 
We now use  Theorem \ref{thm:G} with $G= [(F' H'')^{1/2}]$, which has the property that, for some $K>0$, \[G(r) \geq \dfrac{r}{K^{1/2}} \ \text{ for $0\leq r\leq 2$ \ \ and \ \  $G(r) \geq \frac{2}{K^{1/2}}$ for $r >2$.} \]

Let $T(t)$ be defined as in Lemma \ref{lem:T}, and note that, due to  Theorem \ref{thm:G}, that $\osc(u(\cdot,t)) < 4$ yields 
\[\|G(u_x(\cdot,T(t)))\|_{\infty} \leq  \sqrt{\frac{\osc(u(\cdot,t))}{K}} < \frac{2}{K^{1/2}}, \]
and, hence, 
\[ \|u_x(\cdot, T(t))\|_{\infty} \leq K^{1/2}  \|G(u_x(\cdot, T(t)))\|_{\infty} \leq \sqrt{\osc(u(\cdot,t))}. \]
\vs
In view of  Lemma \ref{lem:T}, the last inequality implies that, for $t$ large enough, 
\begin{equation} \label{eq:utoux}
 \|u_x(\cdot,t)\|_{\infty} \leq { \sqrt{\osc(u(\cdot,t-o(t)))}},
\end{equation} 
 and, 
 in particular,  for all $t \geq T(\tau(u_0,B))$, $\|u_x(\cdot,t)\|_{\infty} \leq 2$.
  \vs
It also follows from Proposition \ref{prop:Gronwall} that,  if $\|u_x(\cdot, t_0)\|_{\infty} \leq  \eta$, then, for all $ t\geq t_0$, 
\begin{equation} \label{eq:L2}
\|u_x(\cdot,t)\|_{L^2} \leq \|u_x(\cdot,0)\|_{L^2} e^{- \frac{\delta}{\lambda_1(1+\eta^2)} (t-t_0)},
\end{equation}
which, combined with \eqref{eq:utoux},  gives  \eqref{eq:uxSmcf}, and with \eqref{eq:Osc2}, implies \eqref{eq:OscSmcf}.
\vs
{  
It only remains to prove \eqref{eq:expdecay}, which now follows immediately from \eqref{eq:L2}, recalling that,by \eqref{eq:uxSmcf} it holds that $\|u_x(\cdot,t)\|_{\infty} \underset{t\to \infty} \to 0$. 
}%
\end{proof}

\begin{remark}\label{rmk:mcf_limit}
Note that \eqref{eq:OscSmcf} may seem expected,  since it is similar to, for example, the  long-time behavior of solutions to the heat equation. This is not, however, the case since,  due to the degeneracy of the mean curvature operator, the solution of the deterministic problem,  that is, when $B\equiv 0$,  exhibits very different behavior.
\vs

Indeed, it is possible to  construct a sequence of solutions $u^R$ to the deterministic mean curvature flow with $\osc(u^R(\cdot,0)) = R$, such that, for $R$ large enough, some $c>0$, and all $t \leq c R$, 
\[ \osc(u^R(\cdot,t)) \geq \frac{R}{2}, \]
which is in sharp contrast with the stochastic setting where, in view of Theorem~\ref{takis68}, the solutions become of order $1$ in a time which only depends logarithmically of the size of the initial condition.
\vs
To construct $u^R$, one can proceed, similarly to \cite{CM04}, by taking as initial condition a graph ``sandwiched''  between two ''grim reaper''-type solutions of the mean curvature flow,  which are finite on a compact interval and move at constant speed of order $R$ downwards (resp. $-R$ and upwards). It then follows from the maximum principle that  the corresponding solution stays in between the grim reapers, and, therefore, its oscillation is greater than $R/2$ for times of order $R$. 
\end{remark}

\begin{remark}
It follows from  Remark \ref{rem:Constants} that in \eqref{eq:OscSmcf} it is possible  to take take any constant $c$ satisfying
\[ c < \min [\frac{2}{e}, \frac{\delta}{\lambda_1}].\]
\end{remark}
\vs
We conclude this section discussing the  example of SPDEs with  nonlinear diffusion degenerating at zero. 

\vs
We recall that the deterministic $p$-Laplace equation 
\begin{equation}\label{takis69}
       \partial_t u 
       = \partial_{x}(u_x)^{[\alpha]} 
     \end{equation}
admits separated variables solutions of the form  $u = g(t)f(x)$ with $g(t) = ((\alpha - 1)t))^{-\frac{1}{\alpha-1}}$, $f(x)+\partial_{x}(f_x)^{\alpha}=0$ and $f$ periodic with average zero. 
It follows that solutions to \eqref{takis69} have decay of order $ t^{-\frac{1}{\alpha-1}}$.  The same optimal order of decay can also be found in signed self-similar solutions (see Hulshof \cite{H89}) on the full space. \vs

In the case of SPDEs, if the order of degeneracy $\a$ of the parabolic part is much larger than the degeneracy $\b$ of the noise in the sense that $\b\le\frac{\a-1}{2}$, then the decay due to the noise is faster than that of the parabolic part, hence, improving the rate of convergence compared to the deterministic case. The parabolic part is only used in the end to improve from a $W^{1,q}$ to a Lipschitz estimate.

\begin{example}\label{eq:polynomial_SPDE}
   We consider nonlinear fluctuating diffusion SPDEs like
     \begin{align*}
       du 
       &= \partial_{x}((u_x)^{[\alpha]})\;dt+ |u_x|^{\beta} \circ dB,
     \end{align*} 
   with $\alpha, \beta > 1$, $B$ Brownian motion, and  \blue{$u_0 \in C(\T)$}. Then, for all $T>0$, 
     \begin{equation} \label{eq:oscuT4}
     \|u_x(\cdot, T)\|_{L^\beta}  \leq C \left( \left(\osc_{0,T}B /C_1\right)^{-\frac{1}{\beta-1}} \right)
     \end{equation}
     and
     $$ \|u_x(\cdot, T) \|_\infty \le C \left( \left(\osc_{0,T-1}B /C_1\right)^{-\frac{1}{\beta-1}} \right)^\frac{\beta}{\alpha+\beta-1} \approx C T^{-\frac{\beta}{2(\beta-1)(\alpha+\beta-1)}}.$$
   Note if $\beta \approx 1$, the decay above exceeds  significantly the optimal deterministic decay $ t^{-\frac{1}{\alpha-1}}$.
\end{example}
Indeed,   Theorem \ref{thm:quantifiedbetter} implies \eqref{eq:oscuT4}, while, in view of  Proposition \ref{prop:deg_diss} with 
  $$G(r)=\int_0^r (w^{\alpha-1}w^{\beta-2})^\frac{1}{2}(w)dw = \int_0^r w^\frac{{\alpha+\beta-3}}{2}\ dw =r^\frac{{\alpha+\beta-1}}{2}, $$
   we find 
    $$G^2( \|u_x(\cdot, T) \|_\infty)\le C\int_\T (u_x(\cdot, T-1))^\beta \,dx.$$
  Therefore,
   $$G^2( \|u_x(\cdot, T) \|_\infty)\le C \left( \left(\osc_{0,T-1}B /C_1\right)^{-\frac{1}{\beta-1}} \right)^\beta,$$  
and,  thus,
     $$ \|u_x(\cdot, T) \|_\infty \le C \left( \left(\osc_{0,T-1}B /C_1\right)^{-\frac{1}{\beta-1}} \right)^\frac{\beta}{\alpha+\beta-1}.$$ 
  
\section{Open questions} \label{sec:open}

We discuss a number of questions about the general theory and the examples presented earlier. 
\vs

Regarding the latter, some of the most intriguing questions are extensions to non-compact domains, the characterization of the random limiting constant, the asymptotic 
behavior of stochastic Hamilton-Jacobi equations with multiple spatially inhomogeneous noises, as well as with non-convex Hamiltonians, and the properties of the stationary solutions $\psi:\T^d\times \R\to \R$.
\vs
The present work unveils an accelerated decay of oscillations for the mean-curvature flow in 2+1 dimensions with spatially homogeneous noise. An extension of these results to higher dimension, in particular, the case of 3+1 dimensions driven by multiple noises is left as an open problem. Moreover, the case of the mean curvature flow with spatially inhomogeneous noise is left untouched, and constitutes a challenging open question.

\appendix

\section{Motivations and examples of nonlinear SPDEs}\label{sec:motivations}

We present  a number of applications and settings that give rise to SPDEs which fall within the framework of this paper. The discussion here  is meant to provide background information and motivation for the examples of SPDEs treated in the main text.  
\vs
We emphasize that the following exposition is meant to be informal and the list of included references is by no means complete. Whenever possible, we refer to monographs and contributions offering exhaustive accounts on the available literature. Listing all of the pertinent literature would be beyond the scope of this appendix.

\subsection[Joint mean field-local interaction limits]{Joint mean field-local interaction limits of particle systems and mean field games with common noise}
\vs

The conditional empirical density measure $\mu^N := \mathcal{L}(\frac{1}{N} \sum_{j=1}^{N}\delta_{X_t^j}\mid \mathcal{W})$ of the mean field-type  interacting particle system
\begin{equation}\label{eqn:MCKV}
dX_t^i = b(X_t^i, \frac{1}{N} \sum_{j=1}^{N}\delta_{X_t^j})dt + \sigma(X^i_t, \frac{1}{N} \sum_{j=1}^{N}\delta_{X_t^j}) \circ dW_t + \alpha(X_t^i, \frac{1}{N} \sum_{j=1}^{N}\delta_{X_t^j})dB_t, 
\end{equation}
converges, in the mean field limit $N\to\infty$, to a solution of the nonlinear, nonlocal, stochastic Fokker Planck equation 
\begin{equation}
\label{target equation}
	\begin{array}{l}
		d\mu
		= [ \partial^2_{i,j}(a^{ij}(x,t,\mu) \mu)
		- \partial_{i} (b^i(x,t,\mu) \mu) ]dt  
		- \partial_{i}(\sigma^{ik}(x,t,\mu)\mu) \circ dW^k_t \ \ \ \ \mu(\cdot,0) = \mu_0,
	\end{array}
\end{equation}
where $a^{ij}=\frac{1}{2}\a^{ik}\a^{jk}$ and $W$ and $B$ are independent Brownian motions. This has been shown for $L^2$-valued solutions by Kurtz and Xiong \cite{KX04} and for measure-valued solutions by Coghi and Gess \cite{CG19}. For further work on this, we refer to the references in these works. The conditional law signifies the fact that $W$ acts on all particles and, hence, it is called common noise. See also \cite{K08} for a motivation of the same class of SPDE arising in statistical mechanics. 
\vs
For an exposition of the same limiting problem in the context of mean field games with common noise, we refer to Carmona and Delarue \cite[Section 2.1.2]{CD18} and Cardaliaguet, Delarue, Lasry and Lions \cite{CDLL19}, and the references therein. 
\vs 

The coefficients $a^{ij}$, $b^i$, $\sigma^{ik}$ depend on the measure $\mu$ in a possibly non-local manner which correspond to nonlocal interactions in the original system \eqref{eqn:MCKV}. 
\vs
It is a natural question to ask what happens when the interaction is localized, that is, when $a^{ij}$, $b^i$, $\sigma^{ik}$ are replaced by a sequence of coefficients $a^{ij,\ve}$, $b^{i,\ve}$, $\sigma^{ik,\ve}$ which converge to coefficients with local dependence on the measure, that is, informally, for  $\rho \in L^1$, 
\begin{align*}
  a^{ij,\ve}(x,t,\rho\ dx)\to a^{ij}(x,t,\rho(x)),\ 
  b^{i,\ve}(x,t,\rho\ dx)\to b^{i}(x,t,\rho(x)),\  
  \sigma^{ik,\ve}(x,t,\rho\ dx)\to \sigma^{ik}(x,t,\rho(x)).
\end{align*}
Such limit then leads, always informally, to   the nonlinear, stochastic Fokker Planck equation
\begin{equation}
\label{target equation_2}
	\begin{array}{l}
		d\rho
		= [ \partial^2_{i,j}(a^{ij}(x,t,\rho) \rho)
		- \partial_{i} (b^i(x,t,\rho) \rho) ]dt 
		- \partial_{i}(\sigma^{ik}(x,t,\rho)\rho) \circ  dW^k_t,
	\end{array}
\end{equation} 
with the coefficients now depending  on $\rho$ in a local manner.
\vs
A relevant example is the case for which the diffusion without common noise, that is, when  $\sigma\equiv 0$ in \eqref{eqn:MCKV}, is reversible with respect to a Gibbs' measure $\frac{1}{Z} e^{-V}$. This corresponds to the choice $b^i(x,t,\mu) = a^{ij}(x,t,\mu)\partial_jV(x,t,\mu)$. In this case, the nonlinear, stochastic Fokker Planck equation \eqref{target equation_2} becomes  
\begin{equation}
\label{target equation_3}
	\begin{array}{l}
		d \rho
		= [\partial_{i}(a^{ij}(x,t,\rho) \partial_j\rho)
		+ \partial_{i} (a^{ij}(x,t,\rho)(\partial_j V)(x,t,\rho) \rho)]dt  
		- \partial_{i}(\sigma^{ik}(x,t,\rho)\rho) \circ  dW^k_t.
	\end{array}
\end{equation} 
We discuss here the special case where the non-local interaction is of  convolution-type, that is, 
\begin{align*}
  &a^{ij,\ve}(x,t,\mu)= (V^{1,\ve} \ast \mu)(x)a^{ij}(x),\ \   b^{i,\ve}(x,t,\mu)= (V^{2,\ve} \ast \mu)(x)b^i(x) \ \ \text{and} \ \ \\\
  &\sigma^{ik,\ve}(x,t,\mu)=(V^{3,\ve} \ast \mu)(x)\sigma^{ik}(x),
\end{align*}
where  $(V^{i,\ve})_{\ep>0}$ is a smooth and compactly supported Dirac family.
\vs
Then \eqref{target equation_3} becomes 
\begin{equation}
\label{target equation_4}
	\begin{array}{l}
		d\rho
		= [\partial_{i}(a^{ij}(x)\rho \partial_j\rho)
		+ \partial_{i} (a^{ij}(x)\partial_j V(x)\rho^2 )]dt  
		- \partial_{i}(\sigma^{ik}(x)\rho^2 \circ  dW^k_t).
	\end{array}
\end{equation} 
and, when $d=1$, $V\equiv 0$ and $\sigma^{ik}(x)=1$,  
 \begin{equation}
\label{target equation_5}
	\begin{array}{l}
		d \rho
		= \partial_x(a(x)\rho \partial_x\rho) dt
		- \partial_x(\rho^2 \circ  dW_t),
	\end{array}
\end{equation} 
an equation to which our methods apply  leading to a quantified convergence estimate for $\rho$ to its mean without any non-degeneracy assumptions on the coefficient $a\ge0$. 
\vs
The well-posedness of solutions to \eqref{target equation_5} for spatially homogeneous $a$ follows from Gess and Souganidis \cite{Ge.So2015}. For inhomogeneous, regular $a$, the well-posedness should follow from a modification of these arguments, but is left as an open problem here.


\subsection{Stochastic front propagation}\label{sec:front_propagation}
Consider an evolving hypersurface $\left(\Gamma(t)\right)_{t \geq 0}$ embedded in $\R^d$ moving with  normal velocity (at point $X(t)$ $\in$ $\Gamma(t)$)
\[ { d  X(t) \cdot  n_{\Gamma}(X(t)) }= -\kappa(X(t)) dt + d \zeta(X(t),t), \]
where $n_{\Gamma}(x)$ is the outer normal at $x \in \Gamma$, $\kappa(x)$ is the scalar curvature of $\Gamma$ at $x$, and $\zeta$ is a stochastic perturbation.
\vs

In  the level-set formulation of this evolution, $\Gamma(t)$ is assumed to be  the zero set at time $t$ of the solution $u$ of  the SPDE
\begin{equation}\label{takis810}
d u= | Du | {\rm div} \left( \frac{Du}{|Du|}\right) dt + |Du | \circ  d\zeta(x,t).
\end{equation}
\vs
For particular cases of the noise $d\zeta$, the SPDE \eqref{takis810} can be studied using stochastic viscosity solutions, see, for example \cite{LS98}, Souganidis and Yip \cite{So.Yi.04}, Souganidis \cite{So.19}, Yip  \cite{Yi.2002} and  Dirr, Luckhaus and Novaga \cite{DLN01}. 
\vs
If  $\Gamma(t)$ is of graph-form, that is, 
\[\Gamma(t) = \left\{ (z, f(z,t)) \; : \; z \in \R^{d-1} \right\},\]
then  $f$  satisfies  the graph  SPDE
\[ d f = \sqrt{1+| Df |^2} {\rm div} \left( \frac{Df}{\sqrt{1+| Df |^2}}\right) dt + \sqrt{1+| Df |^2} \circ d\zeta(z,t,f(z,t)). \]

In the physics literature, Kawasaki and Ohta \cite{KO82} 
derived an equation in the case where $d\zeta$ is space-time white noise acting on the interface $\Gamma$, in which case the  graph SPDE  can be (formally) rewritten as
 \begin{equation} \label{eq:KO}
  \partial_t f = \sqrt{1+| Df |^2} {\rm div} \left( \frac{Df}{\sqrt{1+| Df |^2}}\right) + \left(1+| Df |^2 \right)^{1/4}\circ d \xi(z,t), 
  \end{equation}
where $d\xi$ is a space-time white noise on $\R \times \R^{d-1}$. This can be seen  by checking that
\[ d\zeta(z,t,f(z,t)) := \left(1+| Df(z,t) |^2 \right)^{-1/4} d\xi(z,t) \]
defines a space-time white noise on $\Gamma$, that is,  for any smooth compactly supported function $\phi$, $\int_{\Gamma} \phi \zeta d\mu_{\Gamma}$ is a centered Gaussian with variance given by $\int \phi^2 d\mu_{\Gamma}$. Here the measure $\mu_{\Gamma}$ is the surface measure on $\Gamma$, which can be explicitly written as
\[  \int \psi(x,s) ds d\mu_{\Gamma}(x) =  \int_{\R \times \R^{N-1}}  ds dz \phi(z,s,f(z,s)) \sqrt{1+|Df(z,s)|^2} .\]
Note that \eqref{eq:KO} is purely formal even when $d=2$, since its solution $f$ would be too irregular for the nonlinear gradient term to be well-defined. In fact, it is not clear  that the solution should remain a graph, and it is an open question to which extent one can make sense of this equation rigorously. 
\vs 

Of course, it is possible to  consider more general motions either in the deterministic term or the stochastic perturbation like, for example, purely stochastic velocities of the form 
\[ { d  X(t) \cdot  n_{\Gamma}(X(t)) } =a\left(X(t), n(X(t))\right)) \circ d{B}(t),\]
where $n(x)$ is the outer normal to the point $x \in \Gamma(t)$. This motion   can be interpreted as an interface fluctuating according to a space-homogeneous external factor, but with a local speed  depending on a fixed parameter which can be inhomogeneous and anisotropic.  The corresponding level SPDE is 
\[du=a(x,-\dfrac{Du}{|Du|}) |Du| \circ dB,\]
which can be studied using the theory of stochastic viscosity solutions developed by the last two authors.

\vs
Assuming that the interface is of graph-form, that is,  $\Gamma(t) = \left\{(z,f(z,t))\right\}$, and that the heterogeneity does not depend on the last coordinate,  that is, $a(x, v) = a(z, v)$ where $x=(z,x_n)$,  yields  the SPDE 
\[ d f(t,z) = H(z, D f) \circ dB \] 
with 
\[H(z, D f):= a \left(z, \dfrac{(D_z f, 1) }{\sqrt{1+|Df|^2} }\right )\sqrt{1+|Df|^2}.\]
When the function $H$ is convex in its second variable, which is, for instance, the case for isotropic motions, that is,  $a=a(z)$, then the problem is within  the framework developed in section~\ref{sec:inhomogeneous}.

\subsection{Thermodynamically consistent fluctuation corrections to gradient flows}\label{sec:motivation_GF}

We outline how SPDEs arise from thermodynamically consistent fluctuation corrections to gradient flows. 
\vs
We consider nonlinear diffusion PDEs in 1+1 dimension of the form
\begin{equation}\label{takis100}
  \partial_t \rho 
  = \partial_{xx} \Phi(\rho),
\end{equation}
which come up  in various applications, including interacting particle systems, such as the zero range process (see, for example, Kipnis and Landim \cite{KL99}), geometric PDE, such as the curve shortening and mean curvature flows  (see, for example, Funaki and Spohn \cite{Fu.Sp1997}), 
and in the form of porous medium diffusion equations, see, for example, the monograph by V\'azquez \cite{V07}.  
\vs
Introducing, for some nonlinear, convex function $\Psi$ with antiderivative $[\Psi]$,  the energy $E(\rho)=\int [\Psi](\rho)\,dx$ with functional derivative denoted by $\frac{d}{d\rho} E(\rho)$, and the Onsager operator $M_\rho=B_\rho B_\rho^*$ with $$B_\rho g:=\partial_x \left(\left(\frac{\Phi'(\rho)}{\Psi'(\rho)}\right)^\frac{1}{2}g\right),$$
we can informally rewrite \eqref{takis100}  as the following gradient flow on the space of densities: 
\begin{align*}
  d \rho 
  &= \partial_{xx} \Phi(\rho) 
  = \partial_{x}\left( \frac{\Phi'(\rho)}{\Psi'(\rho)}\partial_{x}\Psi(\rho)\right)
  = M_\rho \frac{d}{d\rho} E(\rho).
\end{align*}

The choice of the energy $E$ and the corresponding function $\Psi$ identifies  a gradient structure on the space of densities.
\vs
Following the general ansatz for thermodynamically consistent fluctuations around gradient flows (see, for example, {\"O}ttinger \cite{ottinger2005beyond}), the corresponding fluctuating gradient flow becomes
\begin{equation}\begin{split}\label{eqn:gen_fluct}
 d \rho 
  &= M_\rho \frac{d}{d\rho} E(\rho) + BdW_t = \partial_{xx} \Phi(\rho) dt +  \partial_x \left(\left(\frac{\Phi'(\rho)}{\Psi'(\rho)}\right)^\frac{1}{2}\circ d {\xi}\right),
\end{split}\end{equation}
with $d\xi$ a space-time white noise. 
\vs

In fact, lattice based particle systems correspond to spatially correlated noise 
\begin{equation}\begin{split}\label{eqn:gen_fluct_2}
  \rho_t 
  &= \partial_{xx} \Phi(\rho)  +  \partial_x \left(\left(\frac{\Phi'(\rho)}{\Psi'(\rho)}\right)^\frac{1}{2}\circ d{\xi^\ve}\right),
\end{split}\end{equation}
where $d\xi^\ve$ is white in time and correlated in space with decorellation length of the order of mesh-size; see, for example, Giacomin, Lebowitz and Presutti \cite{GLP99} and  Mariani \cite{M10} for the case of the simple exclusion process.
\vs
Motivated from the above, we analyze the long-time behavior of parabolic-hyperbolic SPDEs with spatially homogeneous noise $B$, on the level of the anti-derivative $u$, that is, 
\begin{equation}\begin{split}\label{eqn:general_HJB}
 d u
  &= \partial_{x} \Phi(\partial_{x}u) dt +   \left(\frac{\Phi'(\partial_{x}u)}{\Psi'(\partial_{x}u)}\right)^\frac{1}{2}\circ{dB}.
\end{split}\end{equation}
\vs

We next introduce two  particular examples of this general framework. The first  is the stochastic mean curvature flow
\begin{equation}\label{eqn:smc}
  d u=\dfrac{\partial_{xx}u}{1+(\partial_{x}u)^{2}} dt  +\sqrt{1+(\partial_{x}u)^{2}}\circ dB,
\end{equation}
which corresponds to $\Phi(\rho)=\arctan(\rho)$. Then, 
$ \left(\dfrac{\Phi'(\rho)}{\Psi'(\rho)}\right)^\frac{1}{2}=\sqrt{1+\rho^{2}}$, that is, 
$\Psi(\rho)=\frac{1}{2} (\frac{\rho}{1 + \rho^2} + \arctan(\rho))$ in \eqref{eqn:general_HJB}.
\vs
The next example addresses the Kawasaki-Ohta equation (see \cite{KO82}), which constitutes the choice of a thermodynamically consistent noise for the curve shortening flow. Following the framework laid out in the beginning of this subsection, for $d=1$, the PDE for the derivative $\rho=u_x$ of the curve given as the graph of $u$ can be written as a gradient flow 
\begin{align*}
  \partial_t \rho 
  &= \partial_x \left((1+\rho^2)^\frac{1}{2}\ \partial_x \left(\frac{\rho}{\sqrt{1+\rho^2}}\right)\right) = M_\rho \frac{d}{d\rho} E(\rho),
\end{align*}
where $\frac{d}{d\rho} E(\rho)$ is the functional derivative of the energy $E(\rho)=\int (1+\rho^2)^\frac{1}{2}\,dx$ and $M=BB^*$ is the Onsager operator with $B_\rho g=:\partial_x \left((1+\rho^2)^\frac{1}{4}g\right)$. 
\vs

Note that,  with this choice of a gradient structure, the curve shortening flow corresponds to a gradient flow with respect to the arclength as energy $E$. In accordance to \eqref{eqn:general_HJB}, on the level of $u$, the fluctuating stochastic curve shortening flow with spatially homogeneous noise takes the form 
\begin{align*}
   d u
  &= \dfrac{\partial_{xx}u}{1+(\partial_{x}u)^{2}} dt   +  (1+(\partial_xu)^2)^\frac{1}{4} \circ{dB}.
\end{align*}

\bibliographystyle{plain}
\bibliography{HJ}

\begin{flushleft}
\footnotesize \normalfont
\textsc{Paul Gassiat \\
Ceremade, Universit\'e de Paris-Dauphine, PSL University\\
Place du Mar\'echal-de-Lattre-de-Tassigny\\
75775 Paris cedex 16, France} \\
\texttt{\textbf{gassiat@ceremade.dauphine.fr}}
\end{flushleft}

\begin{flushleft}
\footnotesize \normalfont
\textsc{Benjamin Gess\\
Fakult\"at f\"ur Mathematik, Universit\"at Bielefeld\\
Bielefeld, D-33615\\
and \\
Max--Planck--Institut f\"ur Mathematik in den Naturwissenschaften\\ 
Leipzig, D-04103} \\
\texttt{\textbf{bgess@math.uni-bielefeld.de}}
\end{flushleft}

\begin{flushleft}
\footnotesize \normalfont
\textsc{Pierre-Louis Lions\\
Coll\`{e}ge de France,\\
11 Place Marcelin Berthelot, 75005 Paris, \\
CEREMADE, \\
and \\
Universit\'e de Paris-Dauphine,\\
Place du Mar\'echal de Lattre de Tassigny,\\
75016 Paris, France}\\
\texttt{\textbf{lions@ceremade.dauphine.fr}}
\end{flushleft}

\begin{flushleft}
\footnotesize \normalfont
\textsc{Panagiotis E.\ Souganidis\\
Department of Mathematics \\
University of Chicago, \\
5734 S. University Ave.,\\
Chicago, IL 60637, USA}\\
\texttt{\textbf{souganidis@math.uchicago.edu}}
\end{flushleft}

\end{document}